\documentclass[11pt]{amsart} 
\usepackage{amssymb,amscd,amsthm,amsmath,amsfonts}
\usepackage[mathscr]{eucal} 

\newcommand {\norm}[1]{\mbox{$\left\|#1\right\|$}} 
\newcommand {\x}{\times} 
\newcommand {\cs}{\mbox{$C^{*}$-algebra}} 
\newcommand {\css}{\mbox{$C^{*}$-algebras}}

\newcommand {\ov}[1]{\mbox{$\overline{#1}$}}

\newcommand {\C}{\mathbb{C}}

\newcommand {\T}{\mathbb{T}} 
\newcommand {\Z}{\mathbb{Z}} 
\newcommand {\Pos}{\mathbb{P}} 

\newcommand {\al}{\mbox{$\alpha$}} 
 
\newcommand {\bt}{\mbox{$\beta$}} 
 
\newcommand {\Ga}{\mbox{$\Gamma$}} 
\newcommand {\de}{\mbox{$\delta$}} 
\newcommand {\De}{\mbox{$\Delta$}}

\newcommand {\la}{\mbox{$\lambda$}}

\newcommand {\mfL}{\mathfrak{L}} 
 
\newcommand {\mfX}{\mathfrak{X}}

\newcommand {\mfK}{\mathfrak{K}} 
 
\newcommand {\mfH}{\mathfrak{H}}

\newcommand {\bgc}{\begin{center}}

\newcommand {\edc}{\end{center}} 
\newcommand {\be}{\begin{enumerate}} 
\newcommand {\ee}{\end{enumerate}} 
\newcommand {\beqn}{\begin{eqnarray}} 
\newcommand {\eeqn}{\end{eqnarray}} 
\newcommand {\beqns}{\begin{eqnarray*}} 
\newcommand {\eeqns}{\end{eqnarray*}} 
\newcommand {\bq}{\begin{quote}} 
\newcommand {\eq}{\end{quote}} 
 
\newcommand {\bi}{\begin{itemize}} 
\newcommand {\ei}{\end{itemize}} 
\newcommand {\bd}{\begin{description}} 
\newcommand {\ed}{\end{description}} 
 
\newcommand {\lan}{\mbox{$\langle$}} 
\newcommand {\ran}{\mbox{$\rangle$}}

\theoremstyle{plain} 
\newtheorem{theorem}{Theorem}%[section]

\newtheorem{proposition}{Proposition}

\numberwithin{equation}{section}

\begin{document}
\title[A duality theorem for Fourier-Stieltjes algebras]{An operator space duality theorem for the Fourier-Stieltjes algebra of a locally compact groupoid}
\author{Alan L. T. Paterson}
\address{Department of Mathematics\\
Campus Box 395, University of Colorado,\\
Boulder, Colorado 80309-0395, USA}
\email{apat1erson@gmail.com}
\date{December, 2010} 

\maketitle

\section{Introduction}

This paper studies a version of the Fourier-Stieltjes algebra for groupoids.   Before considering this in more detail, and for motivation, let us first consider the (much studied) Fourier-Stieltjes algebra $B(G)$ for a locally compact group.  Specializing even further, let us start with a locally compact abelian group $G$ with dual group $\hat{G}$.  The Fourier algebra, $A(G)$, and the Fourier-Stieltjes algebra $B(G)$ can then be defined in commutative harmonic analysis as $L^{1}(\hat{G})$ and $M(\hat{G})$ respectively.  These, of course, are commutative convolution Banach algebras, and using the inverse Fourier transform, one usually regards $A(G), B(G)$ as commutative Banach algebras of continuous bounded functions on the group $G$ itself under pointwise multiplication and a certain norm (not the supremum norm).  Under this identification, positive measures on $M(\hat{G})$ correspond to positive definite functions on $G$, and these span $B(G)$.  Further, point masses on $\hat{G}$ correspond to the characters on $G$.  The positive definite functions on $G$ in turn correspond, by integration against $L^{1}(G)$-functions, to the states of $C^{*}(G)$.  $B(G)$ as the span of the positive definite functions on $G$ then translates over to $B(G)$ as the span of the states on $C^{*}(G)$, and in fact identifies as a Banach space with the dual $C^{*}(G)^{*}$.  In terms of the original definition, i.e. taking $B(G)=M(\hat{G})$ and using $C_{0}(\hat{G})=C^{*}(G)$, this just amounts to saying that $C_{0}(\hat{G})^{*}=M(\hat{G})$.

Turning to the non-abelian locally compact group case, the set $\hat{G}$ of classes of irreducible representations no longer forms a group, the irreducible representations being usually no longer one-dimensional.  So it is not clear how to define $B(G)$ in terms of some ``$M(\hat{G})$''.   However, Pierre Eymard (\cite{Eymard}) was able to show that much of the theory of $A(G)$ and $B(G)$ goes through in suitable form provided we stay with functions on $G$ and omit any mention of ``$L^{1}(\hat{G})$'' and ``$M(\hat{G})$''.   So $B(G)$ is taken to be the span of the positive definite functions on the group $G$.  (Note that all such functions are continuous and bounded on $G$ - as we will see later, a technical difficulty arises at this point in the locally compact groupoid case since measurable positive definite functions need not be continuous (\cite{Ramwal}).)  

To define the norm for $B(G)$ one regards, as above, positive definite functions on the group as states on 
$C^{*}(G)$ and so can simply identify their span $B(G)$ with the dual space $C^{*}(G)^{*}$.  However, it is not immediately clear from this definition of $B(G)$ what its algebra product is, or how its norm can be defined more concretely (rather than just as a Banach space dual norm).   Another disadvantage of staying with the equality ``$B(G)=C^{*}(G)^{*}$ is that an early result of Walter (\cite{Waldual}) showed that $B(G)$ is associated with completely bounded maps on $C^{*}(G)$, which suggests that there is an operator space side to $B(G)$ whose ramifications are not brought out if we regard it only as a dual Banach space.  Fortunately there  
are three other ways that one can regard the equality, which clarify these issues and which are pertinent for the groupoid case.  

The first way is (at first sight) pointless.  One writes 
\[            B(G)=(\C\otimes C^{*}(G)\otimes \C)^{*}.                \]
However, it turns out that this tensor product formulation, in an operator space context, is fundamental for the study of the groupoid case.  (This was first observed by Jean Renault in \cite{Ren}.)  Indeed, even in the group case, the equality $B(G)=(\C\otimes C^{*}(G)\otimes \C)^{*}$ becomes more informative when it is replaced by the corresponding equality of {\em operator spaces}, precisely, 
$B(G)=(\ov{\C}^{r}\otimes _{h} C^{*}(G) \otimes_{h} \C^{c})^{*}$ where $\C^{c}, \ov{\C}^{r}$ are respectively the Hilbert space $\C$ with its column operator space structure and the conjugate $\ov{\C}$ of $\C$ with its row operator space structure, and the Haagerup tensor product is used.  (A brief introduction to the operator space theory used in the paper is given in the second section.)  In the groupoid case, $\C$ above for the group case gets replaced by an $L^{2}$-space on the unit space of the groupoid.

Reverting back to our discussion of the equality $B(G)=C^{*}(G)^{*}$ for the group case (in terms of ``normal'' Functional analysis rather than operator space theory), the second way gives a formula for calculating the norm of $B(G)$.  (We know what $B(G)$ is as an algebra - the span of the positive definite functions on the group - so the problem is to give a concrete way of calculating the $C^{*}(G)^{*}$-norm of the elements of $B(G)$.)  To do this, it is helpful to realize the elements of $B(G)$ as ``coefficients'' of (unitary) representations of $G$ as follows.  By the Gelfand-Naimark-Segal construction, every positive definite function $\phi$ on $G$ gives rise to a representation $L$ of $G$ on some Hilbert space $\mfH$, and there is a vector $\xi\in \mfH$ such $\phi(g)=\lan L_{g}\xi,\xi\ran$ for all $g\in G$.  More generally, using the polarization identity and direct sums of representations, for any 
$\phi\in B(G)$, there exists a representation $L$ of $G$ on some Hilbert space $\mfH$ and elements $\xi, \eta\in \mfH$ such that for all $g\in G$,
\[      \phi(g) = \lan L_{g}\xi,\eta\ran =(\xi,\eta)(g).                   \]
Considering all possible ways of representing $\phi$ as some $(\xi,\eta)$ gives a way to define the norm on 
$B(G)$.  Indeed, using the Jordan decomposition for continuous linear functionals on a $\cs$, Eymard showed that
\begin{equation}   \label{eq:Eymardnorm} 
\norm{\phi} = \inf\{\norm{\xi}\norm{\eta}: \phi=(\xi,\eta)\}
\end{equation}       
the inf being taken over all ways of writing $\phi=(\xi,\eta)$ over all possible representations of $G$.  This way of looking at $B(G)$ also gives the product of two elements of $B(G)$ by just tensoring representations.

There is yet a third way to regard $B(G)$ and this is as an algebra of completely bounded operators on 
$C^{*}(G)$.   This was proved by Walter in \cite{Waldual}.   Indeed, if $\phi\in B(G)$, then pointwise multiplication by $\phi$, $f\to f\phi$, takes $L^{1}(G)$ to itself.  In fact, this map is continuous under the $C^{*}(G)$-norm restricted to $L^{1}(G)$, and since $L^{1}(G)$ is dense in $C^{*}(G)$,  extends to a bounded linear map $T_{\phi}$ on $C^{*}(G)$.  Walter showed that $T_{\phi}$ in fact is a completely bounded map from $C^{*}(G)$ into $C^{*}(G)$ and the completely bounded norm $\norm{T_{\phi}}_{cb}$ of $T_{\phi}$ equals $\norm{\phi}$. 

We might now ask if the theory of $B(G)$ can be extended to the locally compact groupoid case, and if there are correspondingly three ways of regarding $B(G)$.  Justification for the study of $B(G)$ for a locally compact groupoid $G$ lies in the fact that many $\css$ arise naturally as generated by groupoids, so that for such $\css$, groupoids play the role that the group does for group $\css$.  Examples of this include transformation group $\css$, the operator algebras associated with an equivalence relation, the $\cs$ of the holonomy groupoid of a foliation, and graph $\css$.  Given the importance of $B(G)$ in the study of harmonic analysis on groups, in particular, for duality, it is reasonable to expect an equally fundamental role for $B(G)$ in the groupoid case.   Despite it being ``early days'', there is a small but growing literature on $B(G)$ (and the related space $A(G)$) for locally compact groupoids.  The first study in this direction seems to have been made by Walter (\cite{Waldual}) who examined the case where $G=\{1,2,\ldots ,n\}\x \{1,2,\ldots ,n\}$.  Other papers include \cite{Oty,OtyRam,Pat3,Pat4,Ramwal,Ren}.     

In the case of a group, there is, of course, only one unit - the identity element $e$ - but in the case of a groupoid there are usually many units.  In the case of a second countable, locally compact groupoid $G$ with unit space $X$ and left Haar system $\{\la^{x}\}$, one obtains the $\css$ generated by the convolution algebra $C_{c}(G)$ by combining two ingredients, first a special kind of probability measure $\mu$ on $X$ - called a {\em quasi-invariant} measure - with a measurable Hilbert bundle on which the groupoid acts measurably.  In the group case, there is only one quasi-invariant measure, that of the point mass at $e$, and we don't need to mention it explicitly in group representation theory.  In the groupoid case, one integrates the $G$-action against a kind of ``product measure'' $\la^{\mu}=\int_{X} \la^{x}\,d\mu(x)$ (including a ``modular function'' contribution) to obtain a $^{*}$-representation 
$\pi$ of $C_{c}(G)$ on the Hilbert space of 
$L^{2}$-cross-sections of the Hilbert bundle.  With $\mu$ fixed, the largest $\cs$ that can be obtained in this way will be called $C^{*}(G,\mu)$.  This $\cs$ has to be distinguished from the $\cs$ $C^{*}_{\mu}(G)$ of \cite{Ren} 
(defined in the third section of the present paper) which is usually much larger than $C^{*}(G,\mu)$.  In this paper, we will be concerned only with $C^{*}(G,\mu)$ (though see later in this Introduction).  The ``largest'' $\cs$ obtained in this way, as $\mu$ ranges over the set of quasi-invariant measures on $X$, is the full $\cs$ of the groupoid, $C^{*}(G)$, the canonical $\cs$ associated with $G$.

The two papers that are fundamental for the present work are those of Renault (\cite{Ren}) and of Ramsay and Walter (\cite{Ramwal}), and both define their $B(G)$ in a way similar to that of the group case, i.e. as the span of a sutiable class of positive definite functions on $G$ but with certain identifications determined by the  quasi-invariance.   (The group notion of positive definiteness itself extends simply to the groupoid case.)  Renault in \cite{Ren} fixes a quasi-invariant measure $\mu$ and uses essentially bounded $\la^{\mu}$-measurable positive definite functions, whereas Ramsay and Walter, in \cite{Ramwal} uses bounded {\em Borel} positive definite functions, any two of which are identified if there is a Borel set $N$ off which the two functions are equal and is such that 
$\la^{\mu}(N)=0$ for {\em all} quasi-invariant measures $\mu$ on $X$.  Continuity for positive definite functions is not assumed, as one might have expected, basically because in general there are not enough of them to allow the representation theory of $G$ to be used effectively.  In this paper we work, as does \cite{Ren}, in the context of a second countable locally compact groupoid $G$ with left Haar system $\{\la^{x}\}$ and with a given quasi-invariant measure $\mu$ on the unit space $X$ of $G$, and write $B_{\mu}(G)$ in place of $B(G)$ to emphasize the dependence of that space on $\mu$.  The Borel approach of \cite{Ramwal} is closely related to this but differs in having to consider {\em all} quasi-invariant measures at once rather than fixing on one throughout.  The relation between to the two theories will not be considered here.  

We saw earlier that in the group case, $B(G)$-functions can be represented in the form $(\xi,\eta)$ for $\xi,\eta$ in a representation Hilbert space of the group.  A similar result holds in the groupoid case.  Indeed (\cite[\S 3]{Ramwal}, \cite[1.1]{Ren}) the Hilbert space is replaced by a $G$-Hilbert bundle 
$\mfH=\{\mfH_{x}\}_{x\in X}$ and $L$ is a Borel unitary action of $G$ on $\mfH$.  (In particular, each $L_{g}$ is unitary from $\mfH_{s(g)}$ onto $\mfH_{r(g)}$.)   One defines a $B_{\mu}(G)$-function 
$(\al,\bt):G\to \C$ for two essentially bounded sections $\al, \bt$ of $\mfH$ by taking: 
$(\al,\bt)(g)=\lan L_{g}\al_{s(g)},\bt_{r(g)}\ran$.  As in the group case, every $B_{\mu}(G)$-function can be realized as some 
$(\al,\bt)$ for some $G$-Hilbert bundle $\mfH$.   Furthermore, Renault shows that, also as in the group case, we obtain an involutive Banach algebra structure on $B_{\mu}(G)$ by using pointwise products and by defining for 
$\phi\in B_{\mu}(G)$, 
\[    \norm{\phi} = \inf\{\norm{\al}\norm{\bt}: \phi=(\al,\bt)\}.       \]

This raises the natural question: is this a dual norm?  That is, is there some natural Banach space $Z$ such that $B_{\mu}(G)$ (with its norm) is the Banach space dual of $Z$.  In the group case above, $Z$ would be just 
$C^{*}(G)$.  In the groupoid case, Renault solves this, using operator space theory, with a beautiful idea.  In the group case 
(so $\mu=\de_{e}$) and with $\phi=(\xi,\eta)$ ($\xi,\eta\in H$ as above), we identified $\phi$ as a linear functional 
$\theta_{\phi}$ on $L^{1}(G)$ by setting 
\[  \theta_{\phi}(F)=\lan \pi(F)\xi,\eta\ran          \]
where $\pi$ is the integrated form of the representation $L$ of $G$ on $\mfH$.  Then $\theta_{\phi}$ extends to a continuous linear functional on $C^{*}(G)$ and the identification of $B_{\mu}(G)$ with $C^{*}(G)^{*}$ follows.  When we try to do this in the groupoid case, $\mfH$ becomes a $G$-Hilbert bundle $\mathcal{H}$ (fibered over the unit space $X$) and $\pi$ becomes a unitary $G$-representation $L$ on $\mathcal{H}$.  One integrates up the groupoid representation $L$ to get a representation $\pi$ of the Banach $^{*}$-algebra $L^{I}(G)$ (a groupoid analogue of $L^{1}(G)$ for the group case) on the Hilbert space 
$L^{2}(X,\mathcal{H},\mu)$ of square integrable sections of $H$, and replaces $\xi,\eta$ of the group case by essentially bounded sections $\al, \bt$ of the bundle $\mathcal{H}$.  So the natural way to define a linear functional 
$\theta_{\phi}$ on $L^{I}(G)$ is presumably to take ``$\theta_{\phi}(F)=\lan \pi(F)\al,\bt\ran$''.  But this leads to a difficulty since the $\al, \bt$ in the formula for $\theta_{\phi}$ are essentially bounded elements of the Hilbert space $L^{2}(X,H,\mu)$, and these are special elements of this Hilbert space.  We want to involve the whole of this Hilbert space of $L^{2}$-sections, not just the $L^{\infty}$-ones.  Renault solves this difficulty by multiplying both $\al, \bt$ by functions $a, b$ in $L^{2}(X,\mu)$ and this gives two general $L^{2}$-sections of $\mathcal{H}$ to which we can apply $\pi(F)$.  One obtains a linear functional, not on $C^{*}_{\mu}(G)$, but, taking the $a, b$ into account, on 
$\ov{L^{2}(X,\mu)}\otimes C_{\mu}^{*}(G)\otimes L^{2}(X,\mu)$ by defining 
\[   \theta_{\phi}(\ov{b}\otimes F\otimes a)=\lan \pi(F)(a\al),b\bt\ran          \]
with $\phi=(\al,\bt)$.   To interpret $\ov{L^{2}(X,\mu)}\otimes C_{\mu}^{*}(G)\otimes L^{2}(X,\mu)$, or rather a module version of it,  operator space ideas - in particular, the Haagerup tensor product - are required.  (In the group case, 
$L^2(X,\mu)=\C$ which brings us back to our earlier discussion.)

Also, as commented earlier, we will be using the smaller $\cs$ $C^{*}(G,\mu)$ rather than the algebra $C^{*}_{\mu}(G)$ used by Renault.  The $\cs$ 
$C^{*}(G,\mu)$ relates to the fundamental $\css$ associated with a locally compact groupoid, in particular, to the full $\cs$, $C^{*}(G)$, of the groupoid.  We note that in the $C^{*}(G,\mu)$-context, $C_0(X)$ acts as the diagonal algebra, whereas in the case of $C_{\mu}^{*}(G)$, the diagonal algebra is $L^{\infty}(X,\mu)$.  The author has been unable to relate in a satisfactory way \cite[Theorem 2.1, Proposition 2.2]{Ren} to the main theorem of this paper, Theorem~\ref{th:dualbg}.  (In this regard, see the second Note of the third section and Example 3 of the last section of the paper.)   Theorem~\ref{th:dualbg} generalizes the three ways of looking at $B(G)$ for the group case as discussed earlier in this Introduction.    The proof itself is inspired by the proof of Renault's theorem for $C^{*}_{\mu}(G)$, but the essential ideas are simplified by using module versions of a theorem by Effros and Ruan and the generalized Stinespring theorem.  An effort has been made to give complete details of the proof.   In particular, 
certain technical problems involved in making the proof of Theorem~\ref{th:dualbg} work are resolved in the preliminary propositions, Proposition~\ref{prop:abscont}, Proposition~\ref{prop:lebdecomp} and Proposition~\ref{prop:abscont2}.  
The paper closes by discussing four examples illustrating the theorem.  In particular, the fourth example shows how the complete boundedness theory of Schur products (\cite{Paulsen}) fits into the groupoid framework of the theorem.   

The importance of operator space theory in noncommutative harmonic analysis on a locally compact group $G$ was highlighted by the remarkable result of Ruan, described in \cite[16.2]{EffrosRuan}, that while the Fourier algebra $A(G)$ need not be amenable as a Banach algebra when $G$ is amenable, nevertheless {\em $A(G)$ is operator amenable if and only if $G$ is amenable}.  The work of this paper, complementing the work of Renault in \cite{Ren}, emphasizes the even more fundamental, pervasive, role that operator space theory plays in noncommutative harmonic analysis on locally compact groupoids.   For this reason, a brief survey of the operator space theory that we will need for proving Theorem~\ref{th:dualbg} is covered in the second section.  

The author is grateful to Zhong-Jin Ruan and Roger Smith for generous help with the subject of operator space theory, and especially to Arlan Ramsay for many valuable conversations. 

\section{Preliminaries on operator space theory}

This section discusses the basic ideas and results from operator space theory that we will require later in the paper for application within the groupoid context.  For ease of reading, an attempt has been made to make the discussion as self-contained as possible.
The main references for this section are the books: \cite{EffrosRuan} of Effros and Ruan, \cite{Paulsen} of Paulsen, \cite{Pisier} of Pisier, and \cite{BlecherL} of Blecher and Le Merdy. 

A (concrete) {\em operator space} can conveniently be defined as a subspace $X$ of some $B(\mfH,\mfK)$ (or equivalently, of some $B(\mfH)$) where 
$\mfH,\mfK$ are Hilbert spaces.  Each of the spaces $M_n(X)$ of $n\x n$ matrices with entries in $X$ then has its natural norm when identified with a subspace of $B(\mfH^n,\mfK^n)$: for $T=[T_{i,j}]\in M_{n}(X)$, and an $n$-column vector $[h_{j}]$ in $\mfH^{n}$ we just use matrix multiplication: $(Th)_{k}=\sum_{j=1}^{n}T_{k,j}h_{j}$ to get a column vector in $\mfK^{n}$.   Of course, the spaces of rectangular matrices $M_{m,n}(X)$ also have operator space norms 
$\norm{.}_{m,n}$, and these can be reduced to the square case by adding or dropping rows or columns of zeros.  We write $\norm{.}_{n}=\norm{.}_{n,n}$.  

A remarkable abstract characterization of operator spaces - in which we are given a Banach space $X$ and norms 
$\norm{.}_{n}$ assigned to the spaces $M_{n}(X)$ and you can tell, purely from two axioms for the norms 
$\norm{.}_{n}$, when $X$ is a concrete operator space - was given by Ruan.  This is the {\em representation theorem} for operator spaces (e.g. \cite[Theorem 2.3.5]{EffrosRuan}).  The two axioms are, first, that for $x\in M_{m}(X), y\in M_{n}(X)$, we have 
$\norm{x\oplus y}_{m+n}=\max\{\norm{x}_{m}, \norm{y}_{n}\}$, and second, that for 
$\al,\bt\in M_{n}$ and $x\in M_{n}(X)$, we have 
$\norm{\al x\bt}_{n}\leq \norm{\al}\norm{x}_{n}\norm{\bt}$.   One easily checks that for a concrete operator space, both axioms are satisfied.   While the operator space structure on a Banach space $X$ is given by a sequence of norms $\{\norm{.}_{n}\}$ satisfying Ruan's axioms, it is often helpful to specify this structure by a concrete, Hilbert space realization.

If $A$ is a $\cs$ then it becomes an operator space by representing it faithfully on a Hilbert space.  One gets the same norm on $M_{n}(A)$ whatever choice of representation is made, so this gives a {\em canonical} operator space structure on $A$.   Every Banach space $E$ is an operator space in at least one way: identify $E$ canonically with a subspace of the commutative $\cs$ $C(E^{*}_{1})$, where $E^{*}_{1}$ is the unit ball of $E^{*}$ with the relative weak$^{*}$-topology.

The class of operator spaces $(X,\{\norm{.}_{n}\})$ is a category in a natural way.  To define the morphisms, let $X, Y$ be operator spaces and $\Phi:X\to Y$ be a linear map.  Then in the obvious way, $\Phi$ gives rise to a linear map 
$\Phi_{n}:M_{n}(X)\to M_{n}(Y)$ by applying $\Phi$ to matrix entries. The map $\Phi$ is called {\em completely bounded\/} if $\norm{\Phi}_{cb}=\sup_n\norm{\Phi_n}<\infty$.  The set of completely bounded maps from $X$ to $Y$ is a Banach space $(CB(X,Y), \norm{.}_{cb})$, and this is defined to be $Mor(X,Y)$ in the category of operator spaces.   A completely bounded map $\Phi:X\to Y$ is called {\em completely contractive} if 
$\norm{\Phi}_{cb}\leq 1$ and is called a {\em complete isometry} if each 
$\Phi_{n}:M_{n}(X)\to M_{n}(Y)$ is isometric.

The key to the power of operator space theory lies in the fact that its morphisms, the completely bounded maps, are far more special than just Banach space morphisms (bounded linear maps) - in fact, in certain contexts, they are close to algebraic ``homomorphisms''.  The remarkable theorem below, making sense of the last assertion, is due to Wittstock, Haagerup and Paulsen, after earlier work by Stinespring and Arveson.  (There is an important analogue of Theorem~\ref{th:WHP} for bilinear completely bounded maps due to Christensen, Sinclair, Paulsen and Smith (e.g. \cite[Corollary 9.4.5]{EffrosRuan}, \cite[Theorem 1.5.7]{BlecherL}) but we omit this since it will not be needed in the paper.)

\begin{theorem}   \label{th:WHP}
Let $B$ be a $\cs$, $\mfH$ be a Hilbert space and $\Phi:B\to B(\mfH)$ be a completely bounded map.  Then there exists a Hilbert space $\mfL$, a $^{*}$-representation $\pi$ of $B$ on $\mfL$ and bounded linear operators 
$S,T:\mfH\to \mfL$ such that for all $b\in B$,
\begin{equation}   \label{eq:cbpi} 
\Phi(b)=S^{*}\pi(b)T.
\end{equation}                                
Further, $\norm{\Phi}_{cb}\leq \norm{S}\norm{T}$, and $\mfL, \pi, S$ and $T$ can be taken so that  
$\norm{\Phi}_{cb}=\norm{S}\norm{T}$.  Conversely, any linear map $\Phi:B\to B(\mfH)$ which has the form of the right-hand side of (\ref{eq:cbpi}) is a completely bounded map.  
\end{theorem} 

Operator spaces behave well under natural constructions.  Obviously, every linear subspace $V$ of an operator space $X$ inherits from $X$ an operator space structure.  If $V$ is also closed in $X$, then for each $n$, there is the  natural quotient norm $\norm{.}_{n}$ on $M_{n}(X/V)=M_{n}(X)/M_{n}(V)$, and under these norms, $X/V$ is an operator space.    The Banach space completion of an operator space is an operator space in the natural way (since the closure of a subspace of some $B(\mfH,\mfK)$ is also such a subspace).   If $X,Y$ are operator spaces then $CB(X,Y)$ itself is an operator space, $\norm{.}_{n}$ in this case being the norm obtained by identifying $M_{n}(CB(X,Y))$ with $CB(X,M_{n}(Y))$.  Futher, the dual space $X^{*}$ can be identified with the operator space $CB(X,\C)$, and so is a natural operator space.  (Note (\cite[Corollary 2.2.3]{EffrosRuan}) that for $f\in X^{*}$, 
$\norm{f}_{cb}=\norm{f}$.)                              

A Hilbert space $\mfH$ itself is, of course, not immediately given as an operator space, i.e. as a subspace of some 
$B(\mfH',\mfK')$, but can be made into one in a number of different ways.   Two particular operator space structures for $\mfH$ will concern us: these are the {\em column} Hilbert space $\mfH^{c}$ and the {\em row} Hilbert space $\mfH^{r}$.    In each case, we identify the operator space concretely in the form $B(\mfK,\mfK')$ for appropriate Hilbert spaces $\mfK, \mfK'$.   First, we identify $\mfH$ with $B(\C,\mfH)$, where $\xi\in \mfH$ is identified with the linear map 
$T_{\xi}$ that sends the scalar $\la$ to  $\la\xi$.    To explain the ``column'' terminology let us suppose that $\mfH$ is finite-dimensional.   Let 
$\{e_{1}, \ldots ,e_{n}\}$ be an orthonormal basis for $\mfH$,  and identify, through this basis, $\mfH$ as 
$\ell^{2}_{n}$ (i.e. $\C^{n}$ with the standard inner product).  Then $T_{\xi}$ is the $n\x 1$ {\em column} matrix
whose $i$th entry is $\lan \xi,e_{i}\ran$.  Further (\cite[p.46]{BlecherL}) if $\mfH, \mfK$ are Hilbert spaces, then 
$CB(\mfH^{c},\mfK^{c})=B(\mfH,\mfK)$ canonically.

Second, since Hilbert spaces are reflexive, we can identify 
$\mfH$ with the dual of $\mfH^{*}$ to obtain the {\em row} Hilbert space $\mfH^{r}$.   In more detail, identify
$\mfH^{*}$ with the conjugate Hilbert space $\ov{\mfH}$ and let $\eta\to \ov{\eta}$ be the canonical conjugation from 
$\mfH$ to $\ov{\mfH}$.   So $\ov{\eta}:\mfH\to \C$ is given by: $\ov{\eta}(\xi)=\lan \xi,\eta\ran$.  Note that scalar multiplication in $\ov{\mfH}$ is given by: for 
$\la\in \C$, $\la.\ov{\eta}=\ov{\ov{\la}\eta}$.  We obtain the {\em row} Hilbert operator space $\mfH^{r}$ by associating any $\xi\in \mfH$ with the map $S_{\xi}\in B(\ov{\mfH},\C)$ that takes $\ov{\eta}\in \ov{\mfH}$ to 
$\lan \xi,\eta\ran$.    Taking $\{\ov{e_{1}}, \ldots ,\ov{e_{n}}\}$ as the (orthonormal) basis for $\ov{\mfH}$, each
$S_{\xi}$ is realized as the $(1\x n)$-{\em row} matrix with entries $\lan \xi,e_{i}\ran$.    This is why $\mfH^{r}$ is called the {\em row} Hilbert space.   If, instead of $\mfH^{r}$, we look at $\ov{\mfH}^{r}$, then we obtain a nice duality relation between that operator space and $\mfH^{c}$.  In fact, since 
$\ov{\mfH}^{r}=B(\ov{\ov{\mfH}},\C)=B(\mfH,\C)$, we obtain   
\begin{equation}   \label{eq:hcstar}
(\mfH^{c})^{*}=\ov{\mfH}^{r}              
\end{equation}
as operator spaces.  While the above discussion of the column and row Hilbert operator spaces was within the finite-dimensional context, it can be carried through for general Hilbert spaces (\cite[1.2.23]{BlecherL}).

We now discuss operator modules for operator algebras.   The discussion is based on the detailed account given in the book by Blecher and LeMerdy  (\cite[Chapter 3]{BlecherL} ).    An earlier account of some of this is also given in the memoir of Blecher, Muhly and Paulsen (\cite[Chapter 2]{BlechMuhPaul}).  These accounts cover the case of general operator algebras, but we will only need to consider the $\cs$ case.
In addition to the operator module concept discussed below, there are two other notions of ``operator module'' that are useful, associated respectively with the Haagerup and the projective tensor products, but we will not have occasion to use them.

So let $A$ be a $\cs$ with, of course, its canonical operator space structure.   Let $X$ be a left Banach $A$-module that is an operator space.  Then $X$ is called a {\em concrete left operator $A$-module} if $X$ is given as a closed linear subspace of some $B(\mfK,\mfH)$ with the operator space structure inherited from $B(\mfK,\mfH)$, and is such that the module multiplication 
- a bilinear map - is given by a $^{*}$-representation $\theta$ of $A$ on the Hilbert space $\mfH$ with 
$\theta(A)X\subset X$.   (So for $T\in B(\mfK,\mfH)$, $a.T= \theta(a)\circ T$.)  An {\em (abstract) left operator 
$A$-module} is an 
$X$ that is completely isometrically $A$-isomorphic to a concrete one.   {\em Right operator $A$-modules} and 
{\em operator $A$-bimodules} are defined in the obvious ways.   It is easy to check that if $\theta$ and 
$\pi$ are $^{*}$-representations of $A$ on the Hilbert spaces $\mfH, \mfK$ respectively, then $B(\mfK,\mfH)$ itself is a concrete operator $A$-bimodule, the module actions being the natural ones in which, for $a,b\in A$ and 
$T\in B(\mfK,\mfH)$, $aTb=\theta(a)\circ T \circ \pi(b)$.    

With $\theta$ as above, the Hilbert space $\mfH$ is a left $A$-module with module action given by: 
$a.\xi=\theta(a)\xi$.    The column Hilbert operator space $\mfH^{c}$ can be realized as a space of rank $1$ operators on $\mfH$ (\cite[p.11]{BlecherL}), and it easily follows that $\mfH^{c}$ is a left operator $A$-module 
(\cite[Proposition 3.1.7]{BlecherL}).   Further, there is a dual right module action of $A$ on 
$\mfH^{*}=\ov{\mfH}$ given by: $\ov{\eta}.a=\ov{\pi(a^{*})\eta}$, and $\ov{\mfH}^{r}$ is a right operator $A$-module.

Given operator $A$-bimodules $W,Z$, let $CB_{A,A}(Z,W)$, which we shall usually abbreviate to $CB_{A}(Z,W)$, be the operator space (under the structure that it inherits as a subspace of the operator space $CB(Z,W)$) of completely bounded $A$-bimodule maps $\Phi:Z\to W$ (so that a completely bounded map $\Phi:Z\to W$ belongs to $CB_{A}(Z,W)$
if and only if  
\[               a\Phi (z)a'=\Phi (aza')            \]
for all $z\in Z$ and all $a,a' \in A$).  The next result is a module version of Theorem~\ref{th:WHP}.

\begin{proposition}      \label{prop:ccmodule} 
Let $A$ be a commutative $C^{*}$-subalgebra of the multiplier algebra $M(B)$ of a $\cs$ $B$ (so that $B$ is an operator $A$-bimodule).  Let $\rho:A\to B(\mfH)$ be a representation of $A$ on a Hilbert space 
$\mfH$ (so that $B(\mfH)$ is an operator $A$-bimodule).  Let $\Phi\in CB_{A}(B,B(\mfH))$ be completely bounded.  Then there exists a representation $\pi$ of $B$ on a Hilbert space $\mfL$ (determining canonically a representation $\pi$ of $M(B)$ and hence of $A$ on 
$\mfL$) and bounded left $A$-module maps $S:\mfH\to\mfL, T:\mfH\to\mfL$, such that for all $b\in B$, 
               \[      \Phi (b)=S^{*}\pi(b)T.     \] 
Further, $S$ and $T$ can be taken to satisfy $\norm{\Phi}_{cb}=\norm{S}\norm{T}$.
\end{proposition}
\begin{proof}
By Theorem~\ref{th:WHP}, there exists a representation $\pi$ of $B$ on a Hilbert space $\mfL$ and bounded, linear maps 
$S',T':\mfH\to \mfL$ such that for all $b\in B$,  
$\Phi (b)=(S')^{*}\pi (b)T'$.  By extending $\rho$ to the multiplier algebra of $A$, the Hilbert space $\mfH$ is a left $M(A)$-module.   A general result (unpublished) of Roger Smith for {\em any} $C^{*}$-subalgebra $A$ of $M(B)$ gives that $S',T'$ can be taken to be 
$A$-module maps $S,T$.  In our special situation a simple averaging procedure gives the result as follows.   Let 
$A^{1}\subset M(B)$ be $A$ if $1\in A$ and $A + \C 1$ otherwise.   Let $\mathcal{U}$ be the unitary group of $A^{1}$.  Since $\Phi$ is a left $A$-module map, for every $u\in \mathcal{U}$, $b\in B$, $\Phi(ub)=u\Phi(b)$.  (In the following, we write $u$ in place of 
$\rho(u), \pi(u)$ for ease of notation.)  So 
$u(S')^{*}u^{-1}\pi(b)T'=(S')^{*}\pi(b)T'$, and with $m$ an invariant mean on $\mathcal{U}$, 
\[  \Phi(b)=(\int u(S')^{*}u^{-1}\,dm(u))\pi(b)T'=S^{*}\pi(b)T'        \]
where $u_{1}Su_{1}^{-1}=S$ for all $u_{1}\in \mathcal{U}$.  Since $\mathcal{U}$ spans $A^{1}$, $S$ is an $A$-module map.  Similarly, using the fact that $\Phi$ is a right $A$-module map, $T=\int u(T')u^{-1}\,dm(u)$ is also an $A$-module map and $\Phi(b)=S^{*}\pi(b)T$.  If we choose $S', T'$ so  that 
$\norm{\Phi}_{cb}=\norm{S'}\norm{T'}$, then the same holds with $S,T$ in place of $S',T'$ since 
$\norm{S}\leq \norm{S'}, \norm{T}\leq \norm{T'}$.
\end{proof}

As in Banach space theory, tensor products play an important role in operator space theory.  For our purposes, we will only require the Haagerup tensor product.  There are other important and interrelated operator space tensor products - see, for example, \cite[Part 2]{EffrosRuan}, \cite[1.5]{BlecherL}.  (In particular, in the context of Theorem~\ref{th:dual} below, the Haagerup tensor product is closely related to the projective tensor product.)  So let $X, Y$ be operator spaces.  For $v\in M_{n,r}(X), w\in M_{r,n}(Y)$, define $v\odot w\in M_{n}(X\otimes Y)$ in terms of ``matrix multiplication'':
\[    (v\odot w)_{i,j}=\sum_{k=1}^{r}v_{i,k}\otimes w_{k,j}.              \]
The norm for $u\in M_n(X\otimes Y)$ is defined:  
\[\norm{u}_h^{n}=\inf\{\norm{v}\norm{w}:u=v\odot w,v\in M_{n,r}(X),w\in M_{r,n}(Y),r\geq 1\}.\]
(It is not hard to show that for every $u$ above, there always is a $v$ and $w$ such that $u=v\odot w$.)
One can show that $X\otimes Y$ under the norms $\norm{.}_{h}^{n}$ satisfies Ruan's axioms, and so is an operator space, and so therefore is its completion $X\otimes_{h} Y$.  This completion is called the {\em Haagerup tensor product} of $X$ and $Y$.  The theory of Haagerup tensor products is described in detail in \cite[Chapter 9]{EffrosRuan}.  Some of the remarkable properties of this tensor product are (\cite[9.2]{EffrosRuan}) that it is associative, that tensor products of complete contractions remain complete contractions, and that it is both projective and injective in the appropriate senses.    

Tensor products are a means for linearizing bilinear maps; the bilinear maps that are linearized by the Haagerup tensor product are those that are {\em completely bounded}.  To define this notion, let $X, Y, Z$ be operator spaces, and let $\phi:X\x Y\to Z$ be a bilinear map.  The norm $\norm{\phi}$ of $\phi$ is defined: 
$\norm{\phi}=\sup\{\norm{\phi(x,y)}: x\in X, y\in Y,\norm{x}\leq 1, \norm{y}\leq 1\}$.  
For each $n$, define a bilinear map $\phi_{n}:M_{n}(X)\x M_{n}(Y)\to 
M_{n}(Z)$, also in terms of ``matrix multiplication'', by setting 
\begin{equation}   \label{eq:phin}
\phi_{n}(u,v)_{i,j}=\sum_{k=1}^{n}\phi(u_{i,k},v_{k,j}).  
\end{equation}
Let $\norm{\phi}_{cb}=\sup_{n\geq 1}\norm{\phi_{n}}$.  (The norms on $M_{n}(X), M_{n}(Y)$ are given, of course, by the operator space structures on $X$ and $Y$.)   Then $\phi$ is called {\em completely bounded} if 
$\norm{\phi}_{cb}<\infty$.   These are the bilinear maps that determine in the natural way the elements of 
$CB(X\otimes_{h} Y,Z)$, and vice versa, corresponding norms being the same.  The bilinear map $\phi$ is said to be 
{\em completely contractive} if $\norm{\phi}_{cb}\leq 1$. 

The next fundamental result, due to Effros and Ruan (\cite[Proposition 9.3.3]{EffrosRuan}), is proved using the equality of the Haagerup and projective tensor products on certain tensor products one of whose factors is a column/row Hilbert space (e.g. \cite{EffrosRuan1991}, \cite{BlechPaul}, \cite{Blech2}, \cite[Proposition 9.3.1]{EffrosRuan}, \cite[Proposition 1.5.14]{BlecherL}).

\begin{theorem}  \label{th:dual} 
Let $\mfH,\mfK$ be Hilbert spaces and $Z$ be an operator space.  Then 
\[     (\ov{\mfH}^{r}\otimes_h Z\otimes_h \mfK^c)^{*}\cong CB(Z,B(\mfK,\mfH))    \]
completely isometrically as operator spaces.  
This identification associates a continuous linear functional $\theta$ on $\ov{\mfH}^{r}\otimes_hZ\otimes_h\mfK^c$ with a mapping $T_{\theta}\in CB(Z,B(\mfK,\mfH))$ given by:  
\begin{equation}\label{eq:etastar}
\lan T_{\theta}(z)(\xi),\eta\ran=\theta(\ov{\eta}\otimes z\otimes\xi )
\end{equation}
for $\ov{\eta}\in \ov{\mfH}^{r},z\in Z,\xi\in\mfK^c$.  
\end{theorem}

We now describe the theory of module Haagerup products developed in \cite[3.4]{BlecherL} and 
\cite[Chapter 2]{BlechMuhPaul}.   Let $A$ be an operator algebra, $X$ a right operator $A$-module and $Y$ a left operator $A$-module.   As noted earlier, the Haagerup tensor product $X\otimes_{h} Y$ linearizes completely bounded bilinear maps on $X\x Y$.   The {\em module} Haagerup tensor product $X\otimes_{hA} Y$, then, should linearize those completely bounded bilinear maps $\phi$ that respect the module actions of $A$, i.e. such that $\phi(xa,y)=\phi(x,ay)$ for all $x\in X, y\in Y$ and $a\in A$.    It can be defined in a natural, universal way.  
Concretely, $X\otimes_{hA} Y$ is realized as the quotient $(X\otimes_hY)/N$ where $N$ is the closure of the subspace of $X\otimes_h Y$ spanned by tensors of the form $xa\otimes y-x\otimes ay$.  The module Haagerup tensor product is itself an operator space since quotients of operator spaces are operator spaces.

If now $Y$ is an operator $A$-bimodule, $X$ a right operator $A$-module and $Z$ a left operator $A$-module,then 
$X\otimes_{hA} Y$ is a right operator $A$-module under the natural action induced by: 
$(x\otimes  y)a=x\otimes (ya)$, and similarly, $Y\otimes_{hA} Z$ is a left operator $A$-module.   Next the module Haagerup tensor product is - see \cite[Theorem 3.4.10]{BlecherL}) for a more general result - associative in the following sense:  for such $X, Y, Z$, canonically 
\[(X\otimes_{hA}Y)\otimes_{hA}Z\cong X\otimes_{hA}(Y\otimes_{hA}Z)\]
completely isometrically.  We denote this common value by  $X\otimes_{hA}Y\otimes_{hA}Z$.   The following is a special case of 
\cite[Lemma 3.4.6]{BlecherL}.

\begin{proposition}    \label{prop:axa}
Let $A$ be a $\cs$ and $X$ be a left operator $A$-module that is essential (in the sense that $AX=X$).  Then 
$A\otimes_{hA} X$ is completely isometrically isomorphic to $X$ as left operator $A$-modules under the natural map 
$a\otimes x\to ax$.
\end{proposition}

The following module version of Theorem~\ref{th:dual} is a special case of \cite[Corollary 3.5.11]{BlecherL}.

\begin{proposition}   \label{prop:moddual}  
Let $A$ be a $\cs$ with $^{*}$-representations on the Hilbert spaces $\mfH, \mfK$  so that, in particular, $\mfK^{c}$ is a left operator $A$-module and $(\ov{\mfH})^{r}$ is a right operator $A$-module, and $B(\mfK,\mfH)$ is an operator $A$-bimodule.     Let $Z$ be an operator $A$-bimodule.  Then  
\begin{equation}    \label{eq:Adual}
(\ov{\mfH}^{r}\otimes_{hA}Z\otimes_{hA}\mfK^{c})^{*}\cong CB_{A}(Z,B(\mfK,\mfH))       
\end{equation}
completely isometrically, with the same identification $\theta\to T_{\theta}$ as in Theorem~\ref{th:dual}.
\end{proposition}

\section{The groupoid Fourier algebras $B_{\mu}(G)$}

We first survey very briefly the basic theory of locally compact (Hausdorff) groupoids.  (For more information, see, for example,  \cite{rg,Pat2,MuhlyTCU}.)  This can be spelled out as follows.  Let $G$ be a locally compact groupoid with left Haar system $\{\la^x\}_{x\in X}$ where $X=G^0=\{g\in G: g^{2}=g\}$  is the unit space of $G$.   So algebraically, $G$ is a small category with inverses.    The range and source maps from $G$ to $X$ are denoted by $r,s$:  
$r(g)=gg^{-1}, s(g)=g^{-1}g$.    Note that the product $g_{1}g_{2}$ of elements $g_{1}, g_{2}$ of $G$ is defined if and only range and source match up, i.e. $r(g_{2})=s(g_{1})$.    Let 
$G^{x}=\{g\in G: r(g)=x\}$.   Multiplication and inversion in $G$ are continuous, and the range and source maps are both continuous and open.   Turning to the left Haar system $\{\la^{x}\}$, each 
$\la^{x}$ is a positive, regular Borel measure whose support is $G^{x}$, $\{\la^{x}\}$ is invariant in the sense that for 
$g\in G$, $g\la^{s(g)}=\la^{r(g)}$.   Further, we require that for every $F\in C_{c}(G)$ - the space of continuous, complex-valued functions on $G$ with compact support - the function 
$x\to \int_{G^{x}} F(g)\,d\la^{x}(g)$ is continuous.   One writes $\la_{x}=(\la^{x})^{-1}$ defined on 
$G_{x}=\{g\in G: s(g)=x\}$.  The space $C_{c}(G)$ is a convolutive $^{*}$-algebra with operations given by:
\[   F*F'(g)=\int_{G^{x}} F(h)F'(h^{-1}g)\,d\la^{x}(h), \hspace{.2in} F^{*}(g)=\ov{F(g^{-1})}.       \]
A norm $\norm{.}_{I}$ is defined on $C_{c}(G)$ by:
\[   \norm{F}_{I}=\max\{\sup_{x\in X}\int_{G^{x}}\!\!\mid F(g)\mid \,d\la^{x}(g), 
\sup_{x\in X}\int_{G_{x}}\!\!\mid F(g)\mid \,d\la_{x}(g)\} <\infty.             \]
Then $C_{c}(G)$ is a normed $^{*}$-algebra under the $I$-norm.

We next survey some of the basic theory of representations of $G$ on measurable Hilbert bundles.   
For a probability measure $\mu$ on the unit space $X$ of $G$,  we integrate up to get a regular Borel measure 
$\nu=\la^{\mu}$: 
\[      \nu=\int_{X}\la^{x}\, d\mu(x)        \]
interpreted in the natural way.  Let $\nu^{-1}$ be the regular Borel measure on $G$ given by: 
$\nu^{-1}(A)=\nu(A^{-1})$.    Then $\mu$ is called {\em quasi-invariant} if $\nu$ is equivalent to $\nu^{-1}$.    The Radon-Nikodym derivative $d\nu/d\nu^{-1}$ on $G$ is called the {\em modular function} for $\mu$, and is
denoted by $\De$, and the symmetrized version $\nu_{0}$ of $\nu$ is given by: $d\nu_{0}=\De^{-1/2}d\nu$.  (In the locally compact group case, we don't notice the quasi-invariant measure $\mu$ involved in the group's representation theory since there is only one unit in that case - the identity element $e$ - and $\mu$ {\em has} to be the point mass at $e$ and can be left implicit.)       

For the representation theory of the groupoid $G$, we consider $G$-Hilbert bundles.  A {\em G-Hilbert bundle} is a pair $(L,\mathcal{H})$ where $\mathcal{H}$ is a Borel Hilbert bundle and for each $g\in G$, we are given a unitary $L_{g}:\mathcal{H}_{s(g)}\to \mathcal{H}_{r(g)}$ such that the map $g\to L_{g}$ is a groupoid homomorphism which is Borel measurable in the sense that for any pair $v,w$ of Borel sections of 
$\mathcal{H}$, the map $g\to \lan L_{g}v_{s(g)},w_{r(g)}\ran$ is Borel measurable on $G$.  Note that no quasi-invariant measure is involved in the notion of a $G$-Hilbert bundle.  The {\em trivial} $G$-Hilbert bundle (cf. \cite[p.93]{Pat2}) is $X\x \C$ where each 
$L_{g}:\C_{s(g)}\to \C_{r(g)}$ is the identity map on $\C$.   

Let $(L,\mathcal{H})$ be a $G$-Hilbert bundle and $\mu$ a quasi-invariant measure on $G$.   We refer to the triple
$(L,\mathcal{H},\mu)$ as a {\em representation triple}.     Let $L^2(\mathcal{H},\mu )$ be the Hilbert space of square integrable sections of $\mathcal{H}$.  The groupoid representation $L$ integrates up to give a 
$^{*}$-representation, denoted by  $\pi_{L,\mu}$, or simply by $\pi_{L}$ or even $\pi$, of the convolution $^{*}$-algebra $C_{c}(G)$ on $L^2(\mathcal{H},\mu )$ where, for $\xi, \eta\in L^2(\mathcal{H},\mu )$,
\begin{equation}    \label{eq:pirep}
\lan\pi (F)\xi ,\eta\ran=\int F(g)\lan L_{g}\xi_{s(g)},\eta_{r(g)}\ran \,d\nu_0(g).
\end{equation}
This can be conveniently contracted to:
\begin{equation}    \label{eq:pirepsimple}
\pi(F)\xi(x)=\int_{G^{x}} F(g)L_{g}\xi_{s(g)}\De^{-1/2}(g)\,d\la^{x}(g).
\end{equation}

A deep theorem of Renault (\cite{Renrep,MuhlyTCU}) - see \cite[Appendix B]{MuhWill} for a complete proof of the theorem (covering even the locally Hausdorff case) - gives that every $I$-norm continuous representation of $C_{c}(G)$ on a Hilbert space is the integrated form of some representation triple $(L,\mathcal{H},\mu)$.  The measure $\mu$ is determined as follows.  The representation $\pi$ extends in a natural way to a representation of $A=C_{0}(X)$ on 
$L^2(\mathcal{H},\mu)$, and one takes $\mu$ to be a probability measure on the spectrum of $\pi(A)\subset X$ that is 
{\em basic} (\cite[Part 1, Chapter 7]{D1}).  (By basic, one means that a subset $W$ of $X$ is 
$\mu$-null if and only if it is null for every spectral measure $\nu_{\xi,\eta}$ for $\pi(A)$ 
($\xi,\eta\in L^2(\mathcal{H},\mu)$).)  Trivially, basic measures are determined uniquely up to equivalence.   Also, such a measure determines up to isomorphism the Hilbert bundle $\mathcal{H}$ (\cite[Part II, Chapter 6, Theorem 2]{D1}). 

We obtain a $C^{*}$-seminorm $\norm{.}_{\mu}$ on  $C_{c}(G)$ by defining
\begin{equation}   \label{eq:normmu}
      \norm{F}_{\mu}=\sup\norm{\pi_{L}(F)}
\end{equation}                 
the $\sup$ being taken over all representation triples $(L,\mathcal{H},\mu)$ ($\mu$ fixed).   Then 
$C^{*}(G,\mu)$ is just the enveloping $\cs$ associated with $(C_{c}(G),\norm{.}_{\mu})$, i.e. the completion of 
$C_{c}(G)/\ker\norm{.}_{\mu}$ or equivalently, the $\cs$ generated by the image of $C_{c}(G)$ under the direct sum of all such $\pi_{L}$'s.    Using the separability of $C_{c}(G)$ under the $I$-norm, we can realize $C^{*}(G,\mu)$ as the closure of the image $\pi(C_{c}(G))$ for a representation $\pi$ coming from some specific representation triple $(L,\mathcal{H},\mu)$.  When we take the direct sum of representations of $C_{c}(G)$ with the $\mu$'s allowed to vary as well, the $\cs$ that we obtain is the full $\cs$, $C^{*}(G)$, and $C_{c}(G)$ is faithfully embedded in it.     This gives the largest $\cs$ norm on $C_{c}(G)$.  A similar argument to that above shows that there is a special $\mu$ for which this full $\cs$ norm is the same as $\norm{.}_{\mu}$.   So $C^{*}(G)$ is one of these 
$C^{*}(G,\mu)$'s.  (There is also a reduced $\cs$ $C^{*}_{red}(G)$ but this is not usually of the form $C^{*}(G,\mu)$ (since it is associated with particular groupoid representations) and we will not have occasion to use it in this paper.)

\def\esssup{\text{ess sup}}

In \cite{Ren}, Renault used a $\cs$ $C^{*}_{\mu}(G)$ bigger than $C^{*}(G,\mu)$ but defined similarly, which we can formulate as follows.  Let $\mu$ be a quasi-invariant measure on $X$ and $L^{I}(G)$ be the space of Borel measurable functions $F:G\to \C$ such that the maps $x\to \int_{G^{x}}\!\left| F(g)\right| \,d\la^{x}(g)$, 
$x\to \int_{G_{x}}\!\left| F(g)\right| \,d\la_{x}(g)$ are $\mu$-essentially bounded.    Two functions in $L^{I}(G)$ are identified if for each $x$, they agree $\la^{x}$ and $\la_{x}$ almost everywhere.    The norm $\norm{.}_{I}$ is defined on $L^{I}(G)$ by:  
\[   \norm{F}_{I}=\max\{\esssup_{x\in X}\int_{G^{x}}\!\!\left| F(g)\right| \,d\la^{x}(g), 
\esssup_{x\in X}\int_{G_{x}}\!\!\left| F(g)\right| \,d\la_{x}(g)\} <\infty.             \]
Then as for $C_{c}(G)$, $L^{I}(G)$ is a normed $^{*}$-algebra under convolution, and contains $C_{c}(G)$.  Again with the quasi-invariant measure $\mu$ fixed, each representation triple 
$(L ,\mathcal{H},\mu )$ integrates (as above) over $L^{I}(G)$ to give an $L^{I}(G)$-continuous $^{*}$-representation.
Define the $\cs$ seminorm $\norm{.}_{\mu}$ on $L^{I}(G)$ as in (\ref{eq:normmu}) and form the enveloping 
$\cs$ of $(L^{I}(G),\norm{.}_{\mu})$.  This $\cs$ is $C^{*}_{\mu}(G)$, and it contains $C^{*}(G,\mu)$ as a $C^{*}$-subalgebra. 

We now turn to the Fourier-Stieltjes algebras for $G$.   This subject is investigated in the papers 
\cite{Ramwal,Ren}.   With the group case in mind, the natural way to define a Fourier-Stieltjes algebra $B(G)$ for a groupoid is first to define the set $P(G)$ of positive definite functions on $G$, and then define $B(G)$ to be the span of $P(G)$.  One would like to be able to require the functions in $P(G)$ to be continuous as they are in the group case, but as shown in \cite[p.364]{Ramwal}, it can happen that there are continuous functions in $B(G)$ that are {\em not} in the span of the continuous $P(G)$ functions, and there are also completeness problems with $B(G)$ if we restrict to continuous functions only.  Instead, we follow Ramsay and Walter (\cite[Definition 3.1]{Ramwal}) and require the functions $\phi$ in $P(G)$ to be Borel measurable and bounded, with positive definiteness taking the form that for all $x\in X$ and $f\in C_{c}(G)$,
\begin{equation}    \label{eq:posdef} 
\int_{G^{x}}\int_{G^{x}} f(g_{1})\ov{f}(g_{2})\phi(g_{2}^{-1}g_{1})\,d\la^{x}(g_{1})\,d\la^{x}(g_{2})\geq 0.
\end{equation}   
Again following \cite{Ramwal}, we define $B(G)$ to be the span of $P(G)$ (in the vector space of scalar-valued functions on $G$).  It follows from \cite[Lemma 3.2]{Ramwal} (and is easy to check directly) that if $(L,\mathcal{H})$ is a $G$-Hilbert bundle and $\al,\bt$ are bounded Borel sections of $\mathcal{H}$, then the function $\phi$ on $G$, defined by: 
$\phi(g)=\lan L_{g}\al_{s(g)},\bt_{r(g)}\ran$, belongs to $B(G)$.   We write $(\al,\bt)$ for this function $\phi$.  The other direction is more subtle.  It follows from \cite[Theorem 3.5]{Ramwal} that if $\phi\in B(G)$, then there exists a $G$-Hilbert bundle
$(L,\mathcal{H})$ and bounded Borel sections $\al, \bt$ of $\mathcal{H}$ such that
\begin{equation}   \label{eq:phixi} 
  \phi=(\al,\bt)
\end{equation} 
$\la^{\mu}$-a.e. for every quasi-invariant measure $\mu$ on $X$.  (This result that $B(G)$-functions can be taken to be for the form $(\al,\bt)$ is the natural groupoid version of the well-known group result of Eymard that $B(G)$-functions are of the form $g\to \lan L_{g}\al,\bt\ran$ for some unitary representation $L$ of $G$ on a Hilbert space of which $\al, \bt$ are elements, but in the groupoid case, it only applies up to quasi-invariance.)   When $\phi\in P(G)$, we can, also as in the group case, take 
$\al=\bt$.

From the preceding, it is clear that if we want to regard the $B(G)$ functions as of the form $(\al,\bt)$ then we need to identify two such functions that coincide $\la^{\mu}$-a.e. for one or more quasi-invariant measures $\mu$.  
There are two natural possibilities in the present situation.  For the first, we fix a quasi-invariant measure $\mu$ and we identify two functions in $B(G)$ if they agree $\la^{\mu}$-almost everywhere on $G$.
We will call $B(G)$ with this identification $B_{\mu}(G)$ to emphasize its dependence on $\mu$.   This is $B(G)$ as it is studied by Renault in \cite{Ren}, and will be further studied in the present paper.  On the other hand, following Ramsay and Walter (\cite{Ramwal}) we can identify two functions in $B(G)$ if they agree $\la^{\mu}$-almost everywhere for {\em every} quasi-invariant measure $\mu$ on $G$.  This $B(G)$ is intrinsic to the groupoid with no preference given to a specific quasi-invariant measure.  As one might expect, there is a close relationship between $B(G)$ and the 
$B_{\mu}(G)$'s but we will not explore this connection in the present paper.   $B(G)$ (and $B_{\mu}(G)$) are algebras under pointwise operations on $G$.  Indeed, by definition, $B(G)$ is a vector space, and the product of two functions is given by taking tensor products of $G$-representations: if $\phi=(\al,\bt),
\phi'=(\al',\bt')$ in $B(G)$, then $\phi\phi'=(\al\otimes \al',\bt\otimes \bt')$.    
It is shown in \cite{Ramwal,Ren} that $B(G)$ and $B_{\mu}(G)$ are Banach algebras under natural norms.

We will consider only $B_{\mu}(G)$ in this paper.  To define the norm $\norm{.}_{\mu}$ on $B_{\mu}(G)$ used by  Renault (\cite[Proposition 1.4]{Ren}), we first have to define the norm used for bounded Borel sections $\al$ of a Hilbert bundle $\mathcal{H}$ over $X$.   We take 
$\norm{\al}$ to be $\esssup_{x\in X} \norm{\al(x)}$.  (Note that $x\to \norm{\al(x)}$ is a Borel function, and that 
$\esssup$ is taken with respect to the measure $\mu$.)  The norm $\norm{\phi}_{\mu}$ of $\phi\in B_{\mu}(G)$ is the natural analogue of the Eymard norm on $B(G)$ in the group case, and is defined by:
\[   \norm{\phi}_{\mu}=\inf\norm{\al}\norm{\bt},         \] 
the $\inf$ being taken over all possible ways of representing $\phi =(\al ,\bt )$ over all possible representation triples 
$(L,\mathcal{H},\mu)$.   A result of Renault (\cite[Proposition 1.4]{Ren}) - see also \cite[Proposition 5]{Pat4} - 
gives that $B_{\mu}(G)$ is a Banach $^{*}$-algebra under pointwise operations, its involution $\phi\to \phi^{*}$ being given by: 
$\phi^{*}(g)=\ov{\phi(g^{-1}})$.  (We note that 
$(\al,\bt)^{*}=(\bt,\al)$.)   The non-trivial part of Renault's argument for this result lies in showing that $B_{\mu}(G)$ is a Banach space.  His method uses a variant of Paulsen's ``off-diagonalization'' technique in the ``positive definite'' setting for groupoids.   (That $B_{\mu}(G)$ is a Banach space will also be an immediate consequence of Theorem~\ref{th:dualbg} below.)   The next result relates Hilbert bundles and the different $B_{\mu}(G)$'s with respect to absolute continuity.
  
\begin{proposition}    \label{prop:abscont}
Let $K$ be a Hilbert space and $\pi':C^{*}(G)\to B(K)$ be a representation.   Let $\pi'$ be the integrated form of the triple
$(L,\mathcal{H},\mu')$.  Let $\mu$ be a quasi-invariant measure on $X$ with $\mu'\ll \mu$, and $\pi$ be the integrated form of the triple $(L,\mathcal{H},\mu)$.  Then the map $R$, where $R\xi'=\xi'(d\mu'/d\mu)^{1/2}$, is a $C_{0}(X)$-module isometric map from 
$L^{2}(X,\mathcal{H},\mu')$ into $L^{2}(X,\mathcal{H},\mu)$.  Further, for $F\in C^{*}(G)$, 
\begin{equation}   \label{eq:Rpi}
R^{*}\pi(F)R=\pi'(F).
\end{equation}
\end{proposition}
\begin{proof}
It is elementary that $R$ is a $C_{0}(X)$-module isometric map from \\
$L^{2}(X,\mathcal{H},\mu')$ into
$L^{2}(X,\mathcal{H},\mu)$.  To prove (\ref{eq:Rpi}), it is sufficient to prove that for $\xi', \eta'\in L^{2}(X,\mathcal{H},\mu')$,
\begin{equation}   \label{eq:pi'pi}
  \lan \pi(F)\xi,\eta\ran = \lan \pi'(F)\xi',\eta'\ran
\end{equation}   
where $\xi=R\xi', \eta=R\eta'$.  To this end, let $\nu, \nu^{-1}, \De, \nu_{0}$ be as earlier for the representation 
$(L,\mathcal{H},\mu)$, and $\nu',\nu'^{-1}, \De', \nu'_{0}$ be the corresponding measures and function for the representation 
$(L,\mathcal{H},\mu')$.  For the sake of brevity, let $p=d\mu'/d\mu$.   Then $\nu'=\int \la^{x}\,d\mu'(x)=\int \la^{x}p(x)\,d\mu(x)$ so that $d\nu'(g)=p(r(g))d\nu(g)$.  Similarly, 
\[     d((\nu')^{-1})(g)=p(s(g))d(\nu^{-1})(g),       \] 
and so 
\begin{flalign*}
\De'(g) &= d\nu'/d((\nu')^{-1})(g) \\
&= d\nu'/d(\nu)(g)\x d\nu/d(\nu^{-1})(g)\x (d((\nu')^{-1})/d(\nu^{-1})(g))^{-1}\\
&=  p(r(g)p(s(g))^{-1}\De(g).          
\end{flalign*}
So 
\begin{flalign*}
d\nu'_{0}/d\nu_{0}(g) &= \De'(g)^{-1/2}\De(g)^{1/2}(d\nu'/d\nu)(g)\\
&\hspace{-.75in}= [p(r(g))p(s(g)^{-1}\De(g)]^{-1/2}\De(g)^{1/2}p(r(g))\\
&\hspace{-.75in}= p(r(g))^{1/2}p(s(g))^{1/2}).  
\end{flalign*}
So for $F\in C_{c}(G)$,
\begin{flalign*}
\lan \pi'(F)\xi',\eta'\ran &=  \int F(g)\lan L_{g}\xi'_{s(g)},\eta'_{r(g)}\ran\,d\nu_{0}'(g)=\\
&\hspace{-1in} \int F(g)\lan L_{g}\xi_{s(g)},\eta_{r(g)}\ran p(s(g))^{-1/2}
p(r(g))^{-1/2}p(r(g))^{1/2}p(s(g))^{1/2}\,d\nu_{0}(g)\\
&\hspace{-1in}=  \int F(g)\lan L_{g}\xi_{s(g)},\eta_{r(g)}\ran \,d\nu_{0}(g)= \lan \pi(F) \xi,\eta\ran.
\end{flalign*}   
\end{proof}

In the main theorem of the paper below we will be faced with the following situation.  We have two representations 
$\pi, \pi'$ of $C^{*}(G)$ with $\pi'$ factoring through $\pi$.  Let $\mu, \mu'$ be the quasi-invariant measures on $X$ associated respectively with $\pi, \pi'$.  It is tempting to conjecture that $\mu'$ is absolutely continuous with respect to $\mu$.  If that were the case, we could then apply the preceding proposition to replace $\mu'$ by $\mu$ and so simplify the proof of the theorem.   However, the conjecture is false, even in very simple cases.  For example, suppose that $G$ is the unit space groupoid $[0,1]$, $\pi$ the (faithful) multiplication representation of 
$C^{*}(G)=C([0,1])$ on $L^{2}([0,1],\mu)$, where $\mu$ is Lebesgue measure restricted to $[0,1]$, and $\pi'$ the point evaluation at $0$.  Then $\mu'$ is the point mass at $0$, and this is mutually singular with respect to $\mu$.  The following two propositions will enable us to get round this difficulty, and in the proof of the theorem, be able to replace $\mu'$ by a measure that is absolutely continuous with respect to $\mu$.

\begin{proposition}   \label{prop:lebdecomp}
Let $\mu, \mu'$ be quasi-invariant measures on $X$ and let $\mu'=\mu_{0}' + \mu_{1}'$ be the Lebesgue decomposition of $\mu'$ with respect to $\mu$ (so that $\mu_{0}'\perp \mu$ and $\mu_{1}'\ll \mu$).  Then 
$\mu_{0}', \mu_{1}'$ are multiples of quasi-invariant measures on $X$.
\end{proposition}
\begin{proof}
Let $\nu'=\int \la^{x}\,d\mu'(x)$, $\nu_{i}'=\int \la^{x}\,d\mu_{i}'(x)$ for $i=0,1$.  Then $\nu'=\nu_{0}'+\nu_{1}'$.
It is obvious that $\nu_{0}'\perp \nu$ and that  $\nu_{1}'\ll \nu$ so that $\nu_{0}' + \nu_{1}'$ is the Lebesgue decomposition of $\nu'$ with respect to $\nu$.  We have to show that $\nu_{i}'\cong (\nu_{i}')^{-1}$ for $i=0,1$.  To this end, since $\mu,\mu'$ are quasi-invariant, we can write $(\nu')^{-1}=f'\nu', \nu'=(f')^{-1}(\nu')^{-1}$ and 
$\nu^{-1}=f\nu, \nu=f^{-1}\nu^{-1}$ for appropriate Borel functions $f', f$, where for all $x$, 
$0<f'(x), f(x)<\infty$.
Then $\nu'=(f')^{-1}(\nu')^{-1}=(f')^{-1}(\nu_{0}')^{-1} + (f')^{-1}(\nu_{1}')^{-1}$, and  
$(f')^{-1}(\nu_{0}')^{-1}\perp (f')^{-1}\nu^{-1}, (f')^{-1}(\nu_{1}')^{-1}\ll (f')^{-1}\nu^{-1}$.  Since 
$(f')^{-1}\nu^{-1}\sim (f')^{-1}f\nu$, we obtain $(f')^{-1}(\nu_{0}')^{-1}\perp \nu$ and 
$(f')^{-1}(\nu_{1}')^{-1}\ll \nu$.  So $\nu'=(f')^{-1}(\nu_{0}')^{-1} + (f')^{-1}(\nu_{1}')^{-1}$ is the Lebesgue
decomposition of $\nu'$ with respect to $\nu$.   Since $\nu'=\nu_{0}' + \nu_{1}'$ is also the Lebesgue decomposition of $\nu'$ with respect to $\nu$, and the decomposition is unique, we get $(f')^{-1}(\nu_{0}')^{-1}=\nu_{0}'$ and 
$(f')^{-1}(\nu_{1}')^{-1}=\nu_{1}'$.  Hence $\mu_{0}', \mu_{1}'$ are multiples of quasi-invariant measures.
\end{proof} 

\begin{proposition}   \label{prop:abscont2}
Let $\mu$ be a quasi-invariant measure on $X$ and $(L,\mathcal{H},\mu)$ be a representation triple for $G$ with integrated form $\pi:C^{*}(G)\to B(L^{2}(\mathcal{H},\mu))$.   Let $\mu=c_{0}\mu_{0} + c_{1}\mu_{1}$ where $\mu_{0}, \mu_{1}$  are mutually singular, quasi-invariant measures on $X$, and $c_{0}, c_{1}$
are non-negative real numbers.    For $i=0,1$, let $\mathcal{H}_{i}$ be the Hilbert bundle $\mathcal{H}$ with inner product 
$\lan ,\ran_{i,x}$ on each $\mathcal{H}_{x}$ the given scalar product scaled by $c_{i}$.  ($0\mathcal{H}_{x}$ is taken to be the zero Hilbert space.)  For $i=0,1$ let $\pi_{i}$ be the integrated form of the triple $(L,\mathcal{H},\mu_{i})$.    Then
\begin{equation}   \label{eq:dsum} 
L^{2}(\mathcal{H},\mu) = L^{2}(\mathcal{H}_{0},\mu_{0})\oplus L^{2}(\mathcal{H}_{1},\mu_{1}),
\end{equation}
the $L^{2}(\mathcal{H},\mu_{i})$'s are invariant subspaces of $\pi(C^{*}(G))$, for each $i$, $\pi_{i}$ is the subrepresentation of 
$\pi$ given by restriction to $L^{2}(\mathcal{H},\mu_{i})$ and $\pi=\pi_{0}\oplus \pi_{1}$.  
\end{proposition}
\begin{proof}
Since $\mu_{0}\perp \mu_{1}$, there exist disjoint Borel subsets $A, B$ of $X$ such that $X=A\cup B$ and 
$\mu_{0}(B)=0=\mu_{1}(A)$.  It is easy to check (\ref{eq:dsum}).  Explicitly, for 
$\xi\in L^{2}(\mathcal{H},\mu)$ let $\xi_{A}, \xi_{B}$ be respectively the restrictions of 
$\xi$ to $A, B$.  The Hilbert space direct sum decomposition of (\ref{eq:dsum}) is given by the map $\xi\to \xi_{A} + \xi_{B}$.
The decomposition $\pi=\pi_{0}\oplus \pi_{1}$ is trivial if either $c_{0}=0$ or $c_{1}=0$, so that we can suppose that 
$c_{0}>0, c_{1}>0$.

Let $\nu=\int \la^{x}\,d\mu(x)$ and for $i=0,1$, $\nu_{i}=\int \la^{x}\,d\mu_{i}(x)$.  (So $\nu_{0}$ in this proof does not have its usual meaning.)  Then $\nu=c_{0}\nu_{0} + c_{1}\nu_{1}$.  Further, $G$ is the disjoint union $r^{-1}A\cup r^{-1}B$, and 
$\nu_{0}(r^{-1}B)=0=\nu_{1}(r^{-1}A)$.  Further,  $(\nu_{0})^{-1}(s^{-1}B)=\nu_{0}(r^{-1}B)=0$, and since $\nu_{0}\sim \nu_{0}^{-1}$, 
we also have $\nu_{0}(s^{-1}B)=0$.  It follows that both of the sets $s^{-1}A\cap r^{-1}B, r^{-1}A\cap s^{-1}B$
are $\nu$-null.  Also $\nu_{0}$ ``lives on'' $A_{r,s}=r^{-1}A\cap s^{-1}A$, and similarly, 
$\nu_{1}$ on $B_{r,s}=r^{-1}B\cap s^{-1}B$.   Note that $(A_{r,s})^{-1}=A_{r,s}, (B_{r,s})^{-1}=B_{r,s}$.   Let $\De_{0}, \De_{1}$ be the modular functions for $\mu_{0}, \mu_{1}$.  Then for $F\in C_{c}(G)$, 
\[ c_{0}\int F\,d\nu_{0}=\int _{A_{r,s}} F\,d\nu=\int_{A_{r,s}} F\,\De\,d\nu^{-1}= c_{0}\int _{A_{r,s}} F\,\De\,d(\nu_{0})^{-1} \]
so that $\De_{0}=\chi_{A_{r,s}}\De$, and similarly, $\De_{1}=\chi_{B_{r,s}}\De$.  Then using the above,
\[ \lan \pi(F)\xi,\eta\ran = \int F(g)\lan L_{g}\xi_{s(g)},\eta_{r(g)}\ran\De^{-1/2}(g)\,d\nu(g) \]
\begin{multline*}
=\int F(g)[\lan L_{g}(\xi_{A})_{s(g)},(\eta_{A})_{r(g)}\ran + \lan L_{g}(\xi_{A})_{s(g)},(\eta_{B})_{r(g)}\ran\\ 
+\lan L_{g}(\xi_{B})_{s(g)},(\eta_{A})_{r(g)}\ran + \lan L_{g}(\xi_{B})_{s(g)},(\eta_{B})_{r(g)}\ran ]  
[(\De_{0})^{-1/2}(g) + (\De_{1})^{-1/2}(g)]\,d\nu(g)
\end{multline*}
\begin{multline*}
= \int F(g)\lan L_{g}(\xi_{A})_{s(g)},(\eta_{A})_{r(g)}\ran_{0}\De_{0}^{-1/2}(g)\,d\nu_{0}(g) + \\
\int F(g)\lan L_{g}(\xi_{B})_{s(g)},(\eta_{B})_{r(g)}\ran_{1}\De_{1}^{-1/2}(g)\,d\nu_{1}(g).
\end{multline*}  
The proposition now follows.
\end{proof}

Let $A=C_{0}(X)$.  Let $A_{\mu}$ be the image of $A$ as multiplication operators in $B(L^{2}(X,\mu))$.  The operator norm of $f\in A_{\mu}$ is the $L^{\infty}$-norm: $\norm{f}=\norm{f}^{\infty,\mu}$.  Now let $(L ,\mathcal{H},\mu )$ be a representation triple of $G$ and $\pi_{L}$ the integrated form of $L$ on $L^{2}(\mathcal{H},\mu)$.  So $\pi_{L}$ is a homomorphism from $C^{*}(G)$ into $B(L^{2}(\mfH,\mu))$.  We will also write $\pi_{L}:A\to B(L^{2}(\mathcal{H},\mu))$ for the diagonal representation of $A$ on $L^{2}(X,\mu)$: $\pi_{L}(f)\xi(x)=f(x)\xi(x)$ for $x\in X$ and $\xi\in L^{2}(\mathcal{H},\mu)$. We note that this $\pi$ can be regarded as a homomorphism with domain $A_{\mu}$ rather than $A$ - it will not usually be an isomorphism on $A_{\mu}$ since some of the Hilbert spaces $\mathcal{H}_{x}$ can be $0$.  However, in the case of the trivial $G$-Hilbert bundle $(L,X\x \C)$, the representation $\pi_{L}$ associated with the representation triple $(L,X\x \C,\mu)$ is an isomorphism on $A_{\mu}$.  It follows that the image of $A$ in 
$C^{*}(G,\mu)$ is also isomorphic to $A_{\mu}$.   From (\ref{eq:pirepsimple}), we see (cf. \cite[p.59]{rg}) that for $f,f'\in A$ and $F\in C_{c}(G)$, $\pi((f\circ r)F(f'\circ s))=\pi(f)\pi(F)\pi(f')$.  We can regard (an image) of $A_{\mu}$ as contained in the multiplier algebra $M(B)$ of $B=C^{*}(G,\mu)$, and it follows, almost by definition 
(\S 2), that $C^{*}(G,\mu)$ is an operator $A$-bimodule.  Also $\mfH=L^{2}(X,\mathcal{H},\mu)$ is a Hilbert $A$-module - and so, in particular, a left $A$-module - so that from the discussion preceding Proposition~\ref{prop:ccmodule},
$\mfH^{c}$ is a left operator $A$-module and $\ov{\mfH}^{r}$ a right operator $A$-module.  So the module Haagerup tensor product 
\[    \mfX_{\mu}=\ov{L^{2}(X,\mu)}^{r}\otimes_{hA}C^{*}(G,\mu )\otimes_{h A}L^2(X,\mu )^c  \]
in Theorem~\ref{th:dualbg} below makes sense.

The following theorem is inspired by analogous results in \cite{Ren}.

\begin{theorem}     \label{th:dualbg}
Let $\mu$ be a quasi-invariant measure on $X$ and let $B=C^{*}(G,\mu)$.  Then $B_{\mu}(G)$ is a Banach $^{*}$-algebra, the following three Banach spaces are canonically isometrically isomorphic, and the operator spaces in {\em (1)} and {\em (2)} are completely isometrically isomorphic.\\ 
{\em (1)} $\mfX_{\mu}^{*}$;\\
{\em (2)} the space $CB_{A}(B,B(L^2(X,\mu)))$ of completely bounded, $A$-bimodule maps from 
$B$ into the operator algebra of bounded linear maps on $L^2(X,\mu)$;\\
{\em (3)} $B_{\mu}(G)$.
\end{theorem}
\begin{proof}
For the first claim, we know that $B_{\mu}(G)$ is a commutative normed $^{*}$-algebra.  That $B_{\mu}(G)$ is a Banach space (which is originally due to Renault) follows from the rest of the theorem, since $X_{\mu}^{*}$ (or 
$CB_{A}(B,B(L^2(X,\mu))$)  are Banach spaces.  We now prove the rest of the theorem.

By Proposition~\ref{prop:moddual} (with $\mfH=\mfK=L^{2}(X,\mu)$), the operator space $\mfX_{\mu}^{*}$ is completely isometrically isomorphic to the operator space 
$CB_{A}(B,B(L^2(X,\mu)))$ under the map $\theta\to T_{\theta}$ given in (\ref{eq:etastar}).  So (1) is completely isometrically isomorphic to (2).

We now show that the Banach spaces of (2) and (3) are isometrically isomorphic.  Let 
$\Phi\in CB_{A}(B,B(L^2(X,\mu)))$.   By Proposition~\ref{prop:ccmodule} (with $\mfH=L^{2}(X,\mu)$) there exists a Hilbert space $\mfL$ with a representation $\pi'$ of $B$ on $\mfL$ - and also of $A\subset M(B)$ - and $A$-module maps 
$S,T:L^2(X,\mu)\to\mfL$ such that for all $F\in B$, 
\begin{equation}\label{eq:phipi}
\Phi (b)=S^{*}\pi' (F)T.
\end{equation}
Furthermore, we can assume that
\begin{equation}   \label{eq:normst}
\norm{S}\norm{T}=\norm{\Phi}.
\end{equation} 
           
Since $\pi'$ is non-degenerate, $B$ is separable and the ranges of $S, T$ are separable (since $L^{2}(X,\mu)$ is).  So we can take $\mfL$ to be separable.  Then $\pi'$ is the integrated form of some representation triple 
$(L',\mathcal{H}',\mu')$ (and so we can identify $\mfL$ with $L^{2}(\mathcal{H}',\mu')$).  We now want to replace $\mu'$ by a quasi-invariant measure which is absolutely continuous with respect to 
$\mu$.  To this end, let $\mu'=\mu_{0} + \mu_{1}$ be the Lebesgue decomposition of $\mu'$ with respect to $\mu$.  So 
$\mu_{0}\perp \mu$ and $\mu_{1}\ll \mu$.  By 
Proposition~\ref{prop:lebdecomp}, $\mu_{0}=c_{0}\mu_{0}', \mu_{1}=c_{1}\mu_{1}'$ where the $c_{i}$'s are non-negative real numbers and the $\mu_{i}'$'s are quasi-invariant measures.  We will suppose that both $c_{i}$'s are positive, the easier cases where one or other is $0$ being left to the reader.  Then $\mu_{0}'\perp \mu$ and $\mu_{1}'\ll \mu$.  By 
Proposition~\ref{prop:abscont2}, we have the direct sum 
$L^{2}(\mathcal{H}',\mu') = L^{2}(\mathcal{H}'_{0},\mu_{0}')\oplus L^{2}(\mathcal{H}'_{1},\mu_{1}')$.   We now claim that $S(L^{2}(X,\mu))\subset L^{2}(\mathcal{H}'_{1},\mu_{1}')$ (and also, of course, $T(L^{2}(X,\mu))\subset 
L^{2}(\mathcal{H}'_{1},\mu_{1}')$). 

To prove this, let $A, B$ be disjoint measurable subsets of $X$ such that $X=A\cup B$ and 
$\mu_{0}'(B)=0=\mu(A)$.  Since $\mu_{1}'\ll \mu$, we also have $\mu_{1}'(A)=0$.  It is sufficient then to show that 
$\norm{(SF)_{\mid A})}_{2}=0$ (in $L^{2}(\mathcal{H}',\mu')$) for all $F\in L^{2}(X,\mu)$.   This in turn is equivalent to showing that for any compact subset $C$ of $A$, 
$\norm{\chi_{C}SF}^{2}=\int_{C}\norm{SF(x)}^{2}\,d\mu'(x)=0$.  To prove that, let 
$V$ be an open subset containing $C$ and let $f\in C_{c}(X)$ with $0\leq f\leq 1$, $f=0$ outside $V$ and $f=1$ on $C$.
Then $\norm{\chi_{C}SF}\leq \norm{fSF}=\norm{S(fF)}\leq 
\norm{S}(\int \left|f(x)\right|^{2}\norm{F(x)}^{2}\,d\mu(x))^{1/2}
\leq \norm{S} (\int_{V} \norm{F(x)}^{2}\,d\mu(x))^{1/2}$.  Since $\mu(C)=0$, there exists a decreasing sequence of open sets $V_{n}$ in $X$ containing $C$ such that $\mu(V_{n})<1/n$.  The dominated convergence theorem then gives that 
$\int_{V_{n}} \norm{F(x)}^{2}\,d\mu(x)\to 0$, and it follows that $\norm{\chi_{C}SF}^{2}=0$ as required.   

By Proposition~\ref{prop:abscont2}, we can take $\mu'$ to be $\mu'_{1}$.  For the purpose of applying 
Proposition~\ref{prop:abscont}, take $\mathcal{H}=\mathcal{H}'_{1}$ and $L=L'$.  Using Proposition~\ref{prop:abscont2}, (\ref{eq:Rpi}) and (\ref{eq:phipi}), we can replace $(L',\mathcal{H}',\mu', \pi')$ by the quadruple $(L,\mathcal{H},\mu,\pi)$ and $S,T$ by 
$R\circ S, R\circ T$.  In particular, we now have $\mfL=L^{2}(\mathcal{H},\mu)$. 

Since $T$ commutes with the $C_0(X)$ actions, it is decomposable and so is given by a measurable family 
$\{T_x:x\in X\}$ of essentially bounded linear operators with $T_x:\C\to\mfL_x$.  (Here, of course, $L^2(X,\mu )$ trivially disintegrates as $\int\C_x\,d\mu (x)$, and $\mfL_{x}=\mathcal{H}_{x}$.)  Similarly, $S$ is also given by a measurable family 
$\{S_x:x\in X\}$ where each $S_{x}:\C\to\mfL_x$.  Since $\mu$ is a finite measure, 
$1\in L^{\infty}(X,\mu)\subset L^{2}(X,\mu)$.  Then $\bt=S(1)$ and $\al=T(1)$ belong to $L^{2}(\mathcal{H},\mu)$.   We now claim that 
$\al,\bt\in L^{\infty}(\mathcal{H},\mu)$ and that for $a, b\in L^{2}(X,\mu)$, we have $S(b)=b\bt$ and 
$T(a)=a\al$.  (Strictly, we should write $\pi(b)\bt, \pi(a)\al$ in place of $b\bt, a\al$.)

To prove this, since $S$ is a $C_{0}(X)$-module map, we have $S(b)=S(b.1)=bS(1)=b\bt$ for $b\in C_{c}(X)$.  Suppose that $\bt$ does not belong to $L^{\infty}(\mathcal{H},\mu)$.  Then for each positive integer $n$,
$A_{n}=\{x\in X: \norm{\bt(x)}>n\}$ has positive $\mu$-measure.  By regularity, there exists a compact subset $C_n$ of $A_{n}$ with $\mu(C_n)>0$ and an open subset $V_n$ of $X$ containing $C_n$ such that $\mu(V_n)<(3/2)\mu(C_n)$.  Let $a_n\in C_{c}(X)$ be such that $0\leq a_n\leq 1$, $a_n=1$ on $C_{n}$ and $a_n$ vanishes outside $V_{n}$.  Then for all $n$,
\begin{flalign*}
\norm{S}((3/2)\mu(C_n))^{1/2} &\geq \norm{S}(\int \chi_{V_n}(x)\,d\mu(x))^{1/2}\geq 
\norm{S}\norm{a_n}_{L^{2}(X,\mu)}\\
& \hspace{-1.3in}\geq \norm{S(a_n)}\geq (\int_{C_{n}} a_n(x)^{2}\norm{\bt(x)}^{2}\,d\mu(x))^{1/2}\geq n\mu(C_n)^{1/2}.
\end{flalign*} 
This is impossible.   So $\bt, \al\in L^{\infty}(\mathcal{H},\mu)$.   Since $S$ is continuous on, and $C_{c}(X)$ is dense in, $L^{2}(X,\mu)$, we obtain that for all $a, b\in L^{2}(X,\mu)$, $Ta=a\al$ and $Sb=b\bt$ (pointwise multiplication).  Note that $\norm{a\al}_{2}\leq \norm{a}_{2}\norm{\al}, \norm{b\bt}_{2}\leq \norm{b}_{2}\norm{\bt}$, and that 
\begin{equation}   \label{eq:normsxieta}
\norm{T}=\norm{\al}, \norm{S}=\norm{\bt}.
\end{equation}
 
Let $\phi\in B_{\mu}(G)$ be given by: $\phi =(\al ,\bt )$.  Then for 
$F\in C_c(G)$ and $a,b\in L^2(X,\mu )$ and noting that $a\al, b\bt\in L^{2}(\mathcal{H},\mu)$, we have, using 
(\ref{eq:phipi}) and (\ref{eq:pirep}),
\begin{flalign*} 
\lan\Phi (F)a,b\ran &=\lan\pi (F)a\al ,b\bt\ran\\
&\hspace{-1.5in} =\int F(g)\lan L_{g}\al_{s(g)},\bt_{r(g)}\ran a(s(g))\ov{b}(r(g))\,d\nu_0(g).
\end{flalign*}
So 
\begin{equation}\label{eq:phiphi}
\lan\Phi (F)a,b\ran=\int F(g)\phi(g)a(s(g))\ov{b}(r(g))\,d\nu_0(g).
\end{equation}
The function $\phi$ is determined $\la^{\mu}$ a.e. (equivalently $\nu_{0}$-a.e.) since the set of functions on $G$ of the form $F(a\circ s)(\ov{b}\circ r)$ ($a,b\in A)$ equals $C_{c}(G)$.  Since the elements of $B_{\mu}(G)$ are identified $\la^{\mu}$-a.e., $\Phi$ determines the element $\phi\in B_{\mu}(G)$.
The map $\Phi\to\phi$ is obviously linear, and is one-to-one into $B_{\mu}(G)$ since, from (\ref{eq:phiphi}), $\phi$ determines $\Phi$.   To show that the map $\Phi$ is also onto $B_{\mu}(G)$, let 
$\phi_{1}=(\al_{1} ,\bt_{1})\in B_{\mu}(G)$ where $\al_{1}, \bt_{1}$ are $L^{\infty}$-sections for some representation triple $(L_{1},\mathcal{H}_{1},\mu)$.  Let $\pi_{1}$ be the integrated form of $L_{1}$.   Then, from 
(\ref{eq:phiphi}) and the argument leading up to it, we can define the $\Phi_{1}:B\to B(L^2(X,\mu))$, corresponding to $\phi_{1}$, and independently of the choice of $\al_{1}, \bt_{1}$, by:
\[  \lan\Phi_{1}(F)a,b\ran=\lan\pi_{1}(F)a\al_{1} ,b\bt_{1}\ran .   \]
To see that $\Phi_{1}(F)\in B(L^{2}(X,\mu))$, 
\[ \left|\lan \Phi(F)a,b\ran\right| \leq \norm{F}_{C^{*}(G,\mu)}\norm{a\al_{1}}_{2}\norm{b\bt_{1}}_{2}
\leq [\norm{F}_{C^{*}(G,\mu)}\norm{\al_{1}}\norm{\bt_{1}}] \norm{a}_{2}\norm{b}_{2},  \]  
and so 
\begin{equation}   \label{eq:normalbt}
\norm{\Phi}\leq \norm{\al_{1}}\norm{\bt_{1}}.
\end{equation}

We claim that $\Phi_{1}\in CB_{A}(B,B(L^2(X,\mu))$.   To see this, 
by Theorem~\ref{th:WHP}, $\Phi_{1}\in CB(B,B(L^{2}(X,\mu)))$ since it is of the form: $F\to S_{1}^{*}\pi_{1}(F)T_{1}$ where $S_{1}(b)=b\bt_{1}$, $T_{1}(a)=a\al_{1}$.   (We note that $T_{1}, S_{1}$ are bounded, and commute with the $A$-action since, for example, in the $S_{1}$-case, for $f\in A$, $fS_{1}(b)=fb\bt_{1}=S_{1}(fb)$.)  Then 
$\Phi_{1}\in CB_{A}(B,B(L^2(X,\mu))$ since, for example, for $f, f'\in A$,  $F\in C_{c}(G)$, $\Phi_{1}(fFf')=S_{1}^{*}\pi_{1}(f)\pi_{1}(F)\pi_{1}(f')T_{1}=fS_{1}^{*}\pi_{1}(F)T_{1}f'=f\Phi_{1}(F)f'$.

Last, we have to show that $\norm{\Phi}=\norm{\phi}_{\mu}$.  To prove this, by (\ref{eq:normst}) and the definition of $\norm{\phi}_{\mu}$, 
\[\norm{\Phi}=\norm{S}\norm{T}=\norm{\al}\norm{\bt}\geq \norm{\phi}_{\mu}.\]
On the other hand, taking the infimum over all pairs 
$(\al_{1},\bt_{1})$ for which $\phi=(\al_{1},\bt_{1})$ in (\ref{eq:normalbt})
gives $\norm{\Phi}\leq \norm{\phi}_{\mu}$, so that 
$\norm{\Phi}=\norm{\phi}_{\mu}$ as asserted.  
\end{proof}

\noindent
{\bf Notes}\\

\noindent (1)  It follows (as in \cite{Ren}) from the preceding theorem that $B_{\mu}(G)$ can be made into an operator space by transferring to it the operator space structure of $X_{\mu}^{*}$ or 
$CB_{A}(B,B(L^2(X,\mu)))$.  (We get the same operator space structure on $B_{\mu}(G)$ whichever of these two operator spaces we choose.)  Nevertheless, an unsatisfactory feature of the preceding theorem is that while two of the three spaces involved ($X_{\mu}^{*}, CB_{A}(B,B(L^2(X,\mu)))$) are operator spaces {\em a priori}, this is not the case 
for the third space $B_{\mu}(G)$.   In \cite[2.3]{Ren}, Renault sketches an $M_{n}$-vector-valued theory of 
$B_{\mu}(G)$ to obtain a Banach algebra $B_{\mu}(G,M_{n})$ which can be used to give an operator space structure for $B_{\mu}(G)$.  \\

\noindent (2)   Earlier in this section, we defined the $\cs$ $C_{\mu}^{*}(G)$ used by Renault in \cite{Ren}.  Renault gives there a ``$C_{\mu}^{*}(G)$'' version of the preceding theorem, with the $\cs$ $A=C_{0}(X)$ replaced by 
$A'=L^{\infty}(X)$.   The result in \cite{Ren} corresponding to Theorem~\ref{th:dualbg} asserts:
\begin{multline}   
(\ov{L^{2}(X,\mu)}^{r}\otimes_{hA'}C^{*}_{\mu}(G)\otimes_{h A'}L^2(X,\mu )^c)^{*}\\
= CB_{A'}(C^{*}_{\mu}(G),B(L^2(X,\mu)))=
B_{\mu}(G).
\end{multline}              
From their definitions,  $C^{*}(G,\mu)$ is a $C^{*}$-subalgebra of 
$C^{*}_{\mu}(G)$, and since $A=C_{0}(X)\hookrightarrow A'=L^{\infty}(X,\mu)$, the restriction map
\[  R: CB_{A'}(C^{*}_{\mu}(G),B(L^2(X,\mu)))\to CB_{A}(C^{*}(G,\mu),B(L^2(X,\mu)))   \] 
is a norm-decreasing $A$-bimodule map.  So if these two spaces of completely bounded module maps are ``the same'', $R$ has to be a $C_{0}(X)$-module isomorphism.   I do not know if this is true or not.  One can show, as follows, that $R$ is an onto map.  To see this, factorize $\Phi\in CB_{A}(C^{*}(G,\mu),B(L^2(X,\mu)))$  as in Proposition~\ref{prop:ccmodule}, so that for some representation $\pi$ of $C^{*}(G,\mu)$ on a Hilbert space $\mfL$ and $A$-module maps 
$S, T:L^{2}(X,\mu)\to \mfL$, we have, for all $F\in C_{c}(G)$, $\Phi(F)=S^{*}\pi(F)T$.  Arguing as in the proof of Theorem~\ref{th:dualbg}, we can take $\pi$ to be the integrated form of a representation triple $(L,\mfH,\mu)$.   
Integrating up $(L,\mfH,\mu)$ over $L^{I}(G)$ yields a representation $\pi'$ of $C_{\mu}^{*}(G)$ on 
$\mfL$ and also of $L^{\infty}(X,\mu)$ as multipliers (extending the representation of $C_{0}(X)$ on $\mfL$).   Note that $S, T$ are $L^{\infty}(X,\mu)$-module maps since $A_{\mu}$ is strong operator dense in $L^{\infty}(X,\mu)$.  So 
$\Phi'$, given by $\Phi'(F)=S^{*}\pi'(F)T$, belongs to $CB_{A'}(C^{*}_{\mu}(G),B(L^2(X,\mu)))$ and $R(\Phi')=\Phi$.

The $\cs$ $C^{*}_{\mu}(G)$ is, in general, more difficult to determine than $C^{*}(G,\mu)$ (cf. Example 3 of the next section).   A difficulty with $C^{*}_{\mu}(G)$ is that the image of $C_{c}(G)$ is not usually dense in it, and indeed it is not usually separable.  (In particular, disintegration theory does not apply to its general representations.)  

The next section examines some examples of Theorem~\ref{th:dualbg}, and relates them to other operator space theoretic results.

\section{Examples}

%\noindent
%{\bf Examples}\\

\noindent (1) Let $G$ be a locally compact group.   Since the identity element $e$ of $G$ is the only unit in $G$, 
$\mu=\de_{e}$ is the only quasi-invariant measure for $G$. Then $C^{*}(G,\mu)=C^{*}(G)$, and Theorem~\ref{th:dualbg} reduces to:  
\[ (\ov{\C}^{r}\otimes_h C^{*}(G)\otimes_h \C^{c})^{*}=CB(C^{*}(G),\C)=B_{\mu}(G). \]
Examining these three spaces in turn, $\C$ in the first of these is treated as a Hilbert space in the obvious way.  As an operator space, $\C^{c}$ is realized as $B(\C,\C)$, so that its operator space structure is the same as the canonical $\cs$ operator space structure on $\C$.  On the other hand, as in \S 2, $\ov{C}^{r}$ as an operator space is identified with $B(\C,\C)$ so that again this is the same as $\C$ with its canonical $\cs$ operator structure.  By 
Proposition~\ref{prop:axa}, $\ov{C}^{r}\otimes_h C^{*}(G)\otimes_{h}\C^{c}=C^{*}(G)$ so that its dual is 
$C^{*}(G)^{*}$.  For the second of these, $CB(C^{*}(G),\C)=B(C^{*}(G),\C)=C^{*}(G)^{*}$ again.  For the third, 
$B_{\mu}(G)$ is the Banach algebra $B(G)$ with Eymard's norm, and it is a theorem of Eymard that 
$B(G)=C^{*}(G)^{*}$.\\

\noindent (2) Let $G=X$.  So $G$ is a unit groupoid, and every probability measure $\mu$ on $X$ is quasi-invariant.  For convenience we take $\mu$ to have support $X$.  Each $\la^{x}$ is the point mass at $x$, and since the support of 
$\mu$ is $X$, the $I$-norm on $C_{c}(X)$ is just the $\sup$-norm.  It follows that the representations of $G$ are just the continuous representations of $C_{0}(X)$ so that $C^{*}(G,\mu)=C_{0}(X)=A$. Theorem~\ref{th:dualbg} then becomes:
\[  (\ov{L^{2}(X,\mu)}^{r}\otimes_{hA} A\otimes_{hA} L^{2}(X,\mu)^{c})^{*}=CB_{A}(A,L^2(X,\mu))=B_{\mu}(G).   \]
We show directly that each of these spaces is $L^{\infty}(X,\mu)$.

For the first of these, by Proposition~\ref{prop:axa} and the associativity of the module Haagerup tensor product, 
it is just the dual of $\ov{L^{2}(X,\mu)}^{r}\otimes_{hA} L^{2}(X,\mu)$.  Now 
$\ov{L^{2}(X,\mu)}^{r}\otimes_{hA} L^{2}(X,\mu)^{c}
=(\ov{L^{2}(X,\mu)}^{r}\otimes_{h} L^{2}(X,\mu)^{c})/N$ where $N$ is the closure of the subspace spanned by 
elements of the form $\ov{\eta}f\otimes \xi - \ov{\eta}\otimes f\xi$ where $\ov{\eta}\in \ov{L^{2}(X,\mu)}^{r}, 
\xi\in L^{2}(X,\mu)^{c}$.  Note that $\ov{\eta}f=\ov{\ov{f}\eta}$.
By \cite[Proposition 9.3.4]{EffrosRuan},
\[   \ov{L^{2}(X,\mu)}^{r}\otimes_{h} L^{2}(X,\mu)^{c}=\mathfrak{T},   \]
the operator space of trace class operators on $L^{2}(X,\mu)$.  (One gives (\cite[Theorem 3.2.3]{EffrosRuan}) 
$\mathfrak{T}$ the operator space structure that it inherits as the dual of the $\cs$ of compact operators on 
$L^{2}(X,\mu)$.)  This identification sends simple tensors to finite rank operators: 
$\ov{\eta}\otimes \xi\to T_{\eta,\xi}$, where 
$T_{\eta,\xi}(\zeta)=\lan\zeta,\eta\ran \xi$.  We now identify 
$(\ov{L^{2}(X,\mu)}^{r}\otimes_{hA} L^{2}(X,\mu)^{c})^{*}=N^{\perp}$ in $\mathfrak{T}^{*}=B(L^{2}(X,\mu))$.  The duality between 
$\mathfrak{T}$ and $B(L^{2}(X,\mu))$ is given by: 
$\lan S, T_{\eta,\xi}\ran=Tr(ST_{\eta,\xi})=Tr(T_{\eta,S\xi})=\lan S\xi,\eta\ran$.  For $S$ to vanish on $N$, we require
that for all $\xi,\eta\in L^{2}(X,\mu), f\in C_{0}(X)$, 
\[ 0=Tr(S(T_{\eta \ov{f},\xi} - T_{\eta,f\xi}))=\lan \eta\ov{f},S\xi\ran - \lan\eta,Sf\xi\ran =\lan\eta,(fS-Sf)\xi\ran,  \]
i.e. that $S$ should commute with $C_{0}(X)$ regarded as a $\cs$ of multiplication operators on $L^{2}(X,\mu)$..  Taking the weak operator closure of $C_{0}(X)$ in $B(L^{2}(X,\mu))$ gives that $S$ commutes with $L^{\infty}(X,\mu)$, and since $L^{\infty}(X,\mu)$ is a masa in $B(L^{2}(X,\mu))$, we obtain that  
$(\ov{L^{2}(X,\mu)}^{r}\otimes_{hA} L^{2}(X,\mu)^{c})^{*}=L^{\infty}(X,\mu)$ as an operator space.

For the second of these, by a module version of the Arveson-Wittstock Hahn-Banach theorem (\cite[Corollary 3.6.3]{BlecherL}, any 
$\Phi\in CB_{A}(A,B(L^2(X,\mu)))$ extends to a completely bounded $A$-bimodule map $\tilde{\Phi}$ from the unitization $A^{+}$ of $A$ into $B(L^2(X,\mu))$.  Let $T_{\Phi}=\tilde{\Phi}(1)$.  Then $\Phi(a)=aT_{\Phi}=
T_{\Phi}a$ in $B(L^{2}(X,\mu)$.   Then $T_{\phi}$ is uniquely determined by $\Phi$, and as
in the previous paragraph, $T_{\Phi}\in L^{\infty}(X,\mu)$.   Define 
$\Ga:CB_{A}(A,L^2(X,\mu))\to L^{\infty}(X,\mu)$ by:  $\Ga(\Phi)=T_{\Phi}$.  Then $\Ga$ is a linear isometry.  Next, for any $T\in L^{\infty}(X,\mu)$, the map $a\to aT$ belongs to  $CB_{A}(A,B(L^2(X,\mu)))$ using elementary algebra and \cite[Proposition 2.2.6]{EffrosRuan}, and so $\Ga$ is a complete isometry. 

For the third of these, noting that $\la^{\mu}=\mu$, the elements of $P(G)$ are the classes of Borel measurable, bounded, positive functions on $X$, and so the span $B_{\mu}(G)$ of the image of $P(G)$ in $L^{\infty}(X,\mu )$ is just $L^{\infty}(X,\mu)$.  It is easy to check that the $B_{\mu}(G)$-norm is the same as the $L^{\infty}$-norm.\\

\noindent (3) Take the locally compact groupoid $G$ to be the discrete groupoid $\Pos\x \Z$, a bundle of groups each isomorphic to $\Z$, over the set $\Pos$ of positive integers.  Then $X=\Pos$ and $A=c_{0}$.  
Take $\mu=\sum_{n=1}^{\infty}2^{-n}\de_{n}$.  Then $L^{2}(X,\mu)\cong \ell^{2}$.  Also, every measure on $X$ is absolutely continuous with respect to $\mu$, and so all of the representation triples for $G$ can be taken to be of the form $(L,\mathcal{H},\mu)$.  It follows that $C^{*}(G,\mu)=C^{*}(G)$.  Next, $C_{c}(G)$ is just the algebra of complex-valued functions $F$ on $G$ with finite support and with $I$-norm given by: $\norm{F}=\max_{n\in \Pos} \sum_{m\in \Z} \left|F(n,m)\right|$, and so its closure is the commutative Banach $^{*}$-algebra $c_{0}(\ell^{1}(\Z))$ with pointwise convolution multiplication.  The characters of this Banach algebra are of the form $F\to \chi(F(n))$ where $\chi\in \widehat{\Z}=\T$, the circle group, and we obtain that 
$C^{*}(G,\mu)=c_{0}(C(\T))$.
  
Turning to $B_{\mu}(G)$, it is the set of bounded functions $\phi$ on $G$ such that 
$\phi_{n}=\phi_{\mid G^{n}}\in B(G^{n})$.  Then for each $n$,  there exists a Hilbert space 
$\mfH_{n}$ and $\al_{n}, \bt_{n}\in \mfH_{n}$ such that $\phi_{n}=(\al_{n},\bt_{n})$ with 
$\norm{\phi_{n}}$ close to $\norm{\al_{n}}\norm{\bt_{n}}$.   We can take $\al, \bt$ as bounded sections of the Hilbert bundle 
$\{\mfH_{n}\}$ over $\Pos$, and it is simple to check that 
$\norm{\phi}=\sup_{n\geq 1} \norm{\phi_{n}}$ and that 
$B_{\mu}(G)=\ell^{\infty}(B(\Z))=\ell^{\infty}(M(\T))$. 

Theorem~\ref{th:dualbg} then states that
\begin{equation}   \label{eq:c0ct} 
 (\ov{\ell^{2}}^{r}\otimes_{hc_{0}} c_{0}(C(\T))\otimes_{hc_{0}} (\ell^{2})^{c})^{*}=CB_{c_{0}}(c_{0}(C(\T)),B(\ell^{2}))
=B_{\mu}(G). 
\end{equation}
Here, $c_{0}(C(\T))=c_{0}\otimes C(\T)$ is a $c_{0}$-module by multiplying by $c_{0}\otimes 1$.  We now look more closely at the terms in (\ref{eq:c0ct}).   Starting with the middle one, $T:c_{0}(C(\T))\to B(\ell^{2})$ belongs to 
$CB_{c_{0}}(c_{0}(C(\T)),B(\ell^{2}))$ if and only if it is completely bounded and commutes with the multiplication representation of $c_{0}$ on $\ell^{2}$.  Then the range of $T$ commutes also with the masa 
$\ell^{\infty}\subset B(\ell^{2})$, and so can be identified with a subspace of $\ell^{\infty}$.  For $h\in C(\T))$ and 
$i\geq 1$ let $h_{i}\in c_{0}(C(\T))$ be given by: $h_{i}(n)=h$ if $n=i$ and is $0$ otherwise.  Since $T$ is a $c_{0}$-module map, $Th_{i}\in \ell^{\infty}$ has all its components $0$ except possibly for its $i$th entry, $\al_{i}(h)$.
Then $\al_{i}$ is a continuous linear functional on $C(\T)$ and so is determined by a measure $\mu(i)\in M(\T)$.
The map $T\to \{\mu(n)\}$ is a linear isometry onto 
$B_{\mu}(G)=\ell^{\infty}(M(\T))$.  This gives the identity of the middle and final terms in (\ref{eq:c0ct}).

We now sketch how one identifies $B_{\mu}(G)$ with $\mfX_{\mu}^{*}$.   We start with 
$\phi\in \ell^{\infty}(M(\T))=B_{\mu}(G)$.  Then the continuous linear functional $S_{\phi}$ on \\
$\ov{\ell^{2}}^{r}\otimes_{hc_{0}} c_{0}(C(\T))\otimes_{hc_{0}} (\ell^{2})^{c}$ associated with $\phi$ in
(\ref{eq:c0ct}) is given by: 
\[   S_{\phi}(\ov{\eta}\otimes F\otimes \xi)=\sum_{i\geq 1} \ov{\eta_{n}}\phi(n)(F(n))\xi_{n} \]
for $\xi, \eta\in \ell^{2}$ and $F\in c_{0}(C(\T))$.   The map $\theta\to T_{\theta}$ of (\ref{eq:etastar}) in the present case is as follows.  With respect to the standard basis on $\ell^{2}$, 
$T_{\theta}$ is just the diagonal operator whose $n$th diagonal entry is 
$\theta(\ov{e_{n}}\otimes F\otimes e_{n})$.

We now make some comments about $C^{*}_{\mu}(G)$.  The Banach $^{*}$-algebra $L^{I}(G)$ is easily seen to be 
$\ell^{\infty}(\ell^{1}(\Z))$ with pointwise convolution.  We write its elements as sequences
$\{g_{n}\}$ of $\ell^{1}(\Z)$-elements such that the map $n\to g_{n}$ is bounded.   
Taking the $\sup$ norm over the characters of $c_{0}(C(\T))$ (canonically extended to characters of $\ell^{\infty}(C(\T))$) 
identifies $C^{*}_{\mu}(G)$ with the closure in 
$\ell^{\infty}(C(\T))$ of its subalgebra 
$Q=\{\{\widehat{g_{n}}\}: g_{n}\in \ell^{1}(\Z), \sup \norm{g_{n}}_{1}<\infty\}$.  
Of course, $C^{*}_{\mu}(G)$ does contain $C^{*}(G,\mu)=c_{0}(C(\T))$, and also all maps from $\Pos$ to $C(\T)$ that have finite range.  However, {\em $C^{*}_{\mu}(G)$ is a proper subalgebra of $\ell^{\infty}(C(\T))$.}   I am grateful to Colin Graham for suggesting the following argument.  We use a well-known Banach space result: {\em if $X, Y$ are Banach spaces and $T:X\to Y$ is a bounded linear map that is not onto $Y$, then for all $n$, there exists $f_{n}$ in the unit ball $U$ of $Y$ with the property that whenever $g\in X$ is such that $\norm{T(g) - f_{n}}<1/2$, then $\norm{g}>n$.} 
(It can be proved by supposing the contrary and then showing that for every $f\in U$, there exists $g\in X$, the sum of a geometric series in $X$, such that $T(g)=f$, contradicting the fact that $T$ is not onto.)  We apply this result in the case where $X=\ell^{1}(\Z)$, $Y=C(\T)$, and $T=\mathcal{F}$, the Fourier transform $g\to \hat{g}$.  It is well-known (see, for example, \cite[Chapter 5, 4.5]{Reiter}, \cite[(37.19)]{HR2}) that $\mathcal{F}$ has dense range but is not onto.  Then for each $n\geq 1$, there exists $f_{n}$ in the unit ball $U$ of $C(\T)$ with the property that whenever $g\in \ell^{1}(\Z)$ is such that $\norm{\hat{g} - f_{n}}<1/2$, then $\norm{g}_{1}>n$. Let $F\in \ell^{\infty}(C(\T))$ be given by: $F(n)=f_{n}$.   Then by construction 
$\norm{F - \{\widehat{g_{n}}\}}\geq 1/2$ in $\ell^{\infty}(C(\T))$ 
for all $\{\widehat{g_{n}}\}\in Q$ so that $F\notin C_{\mu}^{*}(G)$.\\   

\noindent (4) This example gives a groupoid interpretation of a remarkable result, described below, concerning Schur products.  The result is due to Haagerup (\cite{Haag}).  (See also \cite{Waldual} for an early groupoid approach to Schur products.)  The theory of Schur products and complete boundedness is treated in detail in the book of Paulsen 
(\cite[Chapter 8]{Paulsen}).  For $A\in M_{n}$, define a linear map $S_{A}:M_{n}\to M_{n}$ by: $S_{A}(B)_{i,j}=A_{i,j}B_{i,j}$.   So $S_{A}(B)=A*B$, the Schur product of the matrices $A, B$.  By finite dimensionality, $S_{A}$ is a bounded linear operator on $M_{n}$.   The following definitive theorem is \cite[Theorem 8.7]{Paulsen}.

\begin{theorem}    \label{th:sacb}
The following three statements are equivalent: 
\be
\item $\norm{S_{A}}\leq 1$
\item  $\norm{S_{A}}_{cb}\leq 1$
\item  there exist $2n$ vectors $\xi_{j}, \eta_{i}$ ($1\leq j,i\leq n$) in the unit ball of some Hilbert space such that 
$A_{i,j}=\lan \xi_{j},\eta_{i}\ran$ for all $i,j)$.
\ee
\end{theorem}

We need to reformulate this for $A$'s with $\norm{S_{A}}$ arbitrary.  A straightforward scaling argument, left to the reader, gives the following desired reformulation.

\begin{theorem}    \label{th:sacb2}
The following three numbers are equal: 
\be
\item $\norm{S_{A}}$
\item  $\norm{S_{A}}_{cb}$
\item  $inf_{\xi,\eta,\mfH}\max_{i,j}\norm{\xi_{j}}\norm{\eta_{i}}$ where, for some Hilbert space $\mfH$, 
$\xi=\{\xi_{j}\}_{1\leq j\leq n},\\
 \eta=\{\eta_{i}\}_{1\leq i\leq n}$ belong to $\mfH^{n}$ and are such that 
$A_{i,j}=\lan \xi_{j},\eta_{i}\ran$.
\ee
\end{theorem} 

We now interpret the equivalence of (2) and (3) of this theorem in terms of Theorem~\ref{th:dualbg} above for the groupoid $G=G_n=\{1,2,\ldots ,n\}\x\{1,2,\ldots ,n\}$ (an equivalence relation groupoid).  The product and inverse for $G_{n}$ are respectively given by: $(i,j)(j,k)=(i,k)$ and $(i,j)^{-1}=(j,i)$.  The unit space $X$ is the diagonal $\{(i,i):1\leq i\leq n\}$, which we usually identify with $\{1,2, \ldots ,n\}$.  Then  
$r(i,j)=i, s(i,j)=j$, $G^{i}=\{(i,j):1\leq j\leq n\}$.    We take $\la^{i}$ to be counting measure on $G^{i}$.   
The example $G_{n}$ is discussed in \cite[Example, \S 5]{Pat4} in the context of the continuous Fourier-Stieltjes algebra of a groupoid.  Since $G_{n}$ is discrete, $G_{n}$-Hilbert bundles in the Borel sense of this paper are the same as continuous $G_{n}$-Hilbert bundles in the sense of \cite{Pat4}.   In the latter context, a quasi-invariant measure is not directly involved, but in the present context, we need such a measure $\mu$.  We take
$\mu=\frac{\sum_{i=1}^{n}\de_{i}}{n}$, the sum of the point masses over $i$ divided by $n$.  Then 
$\la^{\mu}=\frac{\sum_{i,j}\de_{(i,j)}}{n}$, which is equivalent to counting measure on $G_{n}$.  Identifying 
$C_{c}(G)$ with $M_{n}$ - in which $f\in C_{c}(G)$ is identified with the matrix which has $f(i,j)$ at the 
$(i,j)$th place - we see that $C^{*}(G,\mu)=M_n$ under the usual matrix multiplication.   Theorem~\ref{th:dualbg} then assumes the form:
\begin{equation}    \label{eq:gn}
(\ov{\ell^{2}_{n}}^{r} \otimes_{h D_{n}} M_{n}\otimes_{h D_{n}} (\ell^{2}_{n})^c)^{*} 
\cong CB_{D_{n}}(M_{n},M_{n})\cong B_{\mu}(G_{n})  
\end{equation}
where $\ell^{2}_{n}$ is just $\C^{n}$ with the usual inner product divided by $n$.   Here, $A=C_{0}(X)$ is identified with  $D_{n}$, the $\C^{*}$-subalgebra of diagonal matrices in $M_{n}$, the multiplier action of $A$ on $M_{n}$ being given by matrix multiplication on the right and left.  

We start by looking at $B_{\mu}(G)$.  Let $\mfH\ne 0$ be a Hilbert space.  Then $\{1,2,\ldots ,n\}\x \mfH$ is a $G$-Hilbert bundle with action given by: $(i,j)(j,\zeta)=(i,\zeta)$.  Fix a unit vector 
$\zeta\in \mfH$ and also fix $i,j$.  Define sections $\al, \bt$ of the bundle by setting 
$\al_{i}=\zeta =\bt_{j}$, all other values of $\al, \bt$ being defined to be $0$.  Then with 
$\phi=(\al,\bt)$, we have $\phi(l,m)=\lan (l,m)(m,\al_{m}),(l,\bt_{l})\ran = \lan \al_{m},\bt_{l}\ran$, which equals
$1$ if $(l,m)=(j,i)$ and is $0$ otherwise.  So the matrix $\phi=e_{j,i}$.   It follows that as a vector space, 
$B_{\mu}(G)=M_{n}$, the space of all functions $F:G\to \C$.  However, multiplication in $B_{\mu}(G)$ is pointwise, so that $B_{\mu}(G)$ acts on itself (as $M_{n}$) by Schur multiplication. 

To determine the norm on $B_{\mu}(G_{n})$, we have to consider an arbitrary $G$-Hilbert bundle $\mathcal{H}$ over 
$X$.   This is just a collection of Hilbert spaces $\{\mfH_{i}\}_{1\leq i\leq n}$ with unitaries 
$L_{i,j}:\mfH_{j}\to \mfH_{i}$ satisfying $L_{i,j}\circ L_{j,k}=L_{i,k}$ and $L_{i,j}^{-1}=L_{i,j}^{*}=L_{j,i}$.   We have to consider sections $\al, \bt$ of this bundle and determine, for a given $\phi\in B_{\mu}(G)$, what is 
$\inf \norm{\al} \norm{\bt}$ over all possible ways of representing $\phi=(\al,\bt)$.  Given such a bundle and such sections $\al, \bt$, let $\mfH=\mfH_{1}$ and let $\xi_{j}=L_{1,j}\al_{j},\eta_{i}=L_{1,i}\bt_{i}$.   Then $\lan \xi_{j},\eta_{i}\ran=\lan L_{1,j}\al_{j},L_{1,i}\bt_{i}\ran=\lan L_{i,j}\al_{j},\bt_{i}\ran=\phi(i,j)$, and 
$\norm{\al}=\max_{j} \norm{\xi_{j}}, \norm{\bt}=\max_{i} \norm{\eta_{i}}$.  So $\norm{\phi}$ is exactly the same as the norm of the matrix $A=\phi$ given in (3) of Theorem~\ref{th:sacb2}.   Conversely, given $\phi=A\in M_{n}$ and vectors $\xi_{j}, \eta_{i}$ as in (3) of Theorem~\ref{th:sacb2}, we can construct a corresponding $G$-representation and sections $\al,\bt$ such that $\phi=(\al,\bt)$.   So (3) defines the norm on $B_{\mu}(G)$.

Turning next to the middle expression of (\ref{eq:gn}), it is easy to prove 
(\cite[Exercise 4.4]{Paulsen}) that $CB_{D_{n}}(M_{n},M_{n})=\{S_{A}:A\in M_{n}\}$, so that the identity of 
$CB_{D_{n}}(M_{n},M_{n})$ with $B_{\mu}(G)$ is equivalent to the identity of (2) and (3) of Theorem~\ref{th:sacb2}.
I do not know of a charaterization of the norm of $\mfX_{\mu}$, the module Haagerup tensor product whose dual is the first space in (\ref{eq:gn}).   The equivalence of (1) and (2) in Theorem~\ref{th:sacb2} follows in complete groupoid generality from \cite[Proposition 2.3]{Ren} - the proof for the special case of $G=G_{n}$ in 
Theorem~\ref{th:sacb2} appears in the first two paragraphs of \cite[Theorem 8.7]{Paulsen}.  (This is also proved in 
\cite[Proposition 8]{Waldual}, where estimates for $\norm{S_{A}}$ are obtained.)

Paulsen (\cite[Corollary 8.8]{Paulsen}) obtains an $\ell^{2}$-version of the above Theorem~\ref{th:sacb} (so that 
$B(\ell^{2})$ replaces $M_{n}=B(\ell^{2}_{n})$ and $A$ is an infinite matrix).   In this case, the groupoid $G$ would be $\Pos\x \Pos$ where $\Pos$ is the set of positive integers.  In a more general context, where $G=X\x X$,  Renault
(\cite[Example 1.3]{Ren}) points out that in this case, $B_{\mu}(G)$ is the space of functions $\phi$ on $G$
for which there exists a Hilbert space $\mfH$ and $\al,\bt\in L^{\infty}(X,\mfH,\mu)$ such that 
$\phi(x,y)=\lan \al(x),\bt(y)\ran$ ($x,y\in X$), and that these functions $\phi$ are the Hilbertian functions as defined by Grothendieck.

\end{document}